\documentclass[12pt,a4paper]{article}
\usepackage[utf8]{inputenc}
\usepackage{amsmath}
\usepackage{amsfonts}
\usepackage{amssymb}
\usepackage{graphicx}
\graphicspath{{C:/Users/Molly/Documents/graphics/png/article/}}
\usepackage{amsthm}

\theoremstyle{plain}
\newtheorem{theorem}{Theorem}[section]
\newtheorem{lemma}{Lemma}[section]
\newtheorem{proposition}{Proposition}[section]
\theoremstyle{definition}
\newtheorem{remark}{Remark}[section]
\newtheorem{example}{Example}[section]
\newtheorem{fact}{Fact}
\usepackage[all]{xy}
\begin{document}

\begin{center}
\begin{LARGE}
\textbf{Connected-Sum Decompositions of Surfaces with Minimally-Intersecting Filling Pairs}
\end{LARGE}
\linebreak

Mark Nieland

Department of Mathematics

State University of New York at Buffalo

232 Capen Hall, Buffalo, NY 14260

USA
\end{center}

\begin{abstract}
Let $ S_g $ be a closed surface of genus $ g $ and let $ (\alpha, \beta) $ be a filling pair on $ S_g $; then $ i(\alpha, \beta) \geq 2g-1 $, where $ i $ is the (geometric) intersection number.  Aougab and Huang demonstrated that (exponentially many) minimally-intersecting filling pairs exist on $ S_g $ when $ g > 2 $ by a construction which produces higher-genus surfaces with filling pairs as connected sums of lower-genus surfaces with filling pairs.  We present a generalization of their construction which provides an explicit, algebraic means of determining the homeomorphism class of the resulting pair, and a criterion for determining when a surface with minimally-intersecting filling pair admits a decomposition as a connected sum.
\end{abstract}
\section{Intoduction}

Let $ S_{g,b} $ be the surface with genus $ g $ and $ b $ boundary components (when $ b=0 $, we denote the surface by simply $ S_g $), and let $ 3g-3+b>1 $.  The \textit{curve complex} of this surface, $ \mathcal{C}(S_{g,b}) $, is a simplicial complex whose vertices are isotopy classes of essential simple closed curves (scc) in $ S_{g,b} $.  The vertices $ \alpha_1, \alpha_2, \ldots, \alpha_{k+1} $ span a $ k $-simplex if $ i(\alpha_m, \alpha_n)=0 $ for each $ m,n \in \lbrace 1,2, \ldots, k+1 \rbrace $, where $ i $ is the geometric intersection number.  In other words, there is a representative curve of each isotopy class such that each pair of curves is disjoint from one another.  We frequently equivocate between ``a vertex of the curve complex," ``an isotopy class of essential scc's," and ``a representative of such a class."  If $ b>0 $, the \textit{arc and curve complex}, $ \mathcal{AC}(S_{g,b}) $, is defined similarly:  its vertices are isotopy classes of essential scc's \textit{and} isotopy classes of essential proper arcs (isotopies do not necessarily fix the endpoints of an arc on $ \partial S_{g,b} $); simplex-spanning is defined in the same way.  Note that (unsurprisingly) the curve complex is a subcomplex of the arc and curve complex.

\indent The number $ \xi(S_{g,b})=3g-3+b $ is referred to as the \textit{complexity} of the surface.  If $ \lbrace \alpha_1, \alpha_2, \ldots, \alpha_n \rbrace $ is a set of pairwise disjoint (nonisotopic) essential scc's on $ S_{g,b} $, then $ n \leq \xi(S_{g,b}) $ (such a set of maximal size is known as a \textit{pants decomposition} of the surface).  We will be interested in the case where $ b=0 $ and $ \xi(S_g) \geq 2 $.

\indent Harvey introduced the curve complex in 1978's \cite{Ha}, and it has become an important tool in the study of the mapping class group (e.g., \cite{MM}), 3-manifolds (e.g., \cite{He}), and Teichm\"uller space (e.g., \cite{F}, and it's original use in \cite{Ha}).

\indent In particular, in the \textit{curve graph}, $ \mathcal{C}^1(S_{g,b}) $ (the 1-skeleton of the curve complex), two vertices are joined by an edge if they are disjoint.  If we define each edge to have length 1, then the curve graph is a metric space, where the distance $ d(\alpha, \beta) $ is the minimum length of an edge-path from $ \alpha $ to $ \beta $ (if such a path exists), or $ \infty $ (if it does not).  Since the curve complex is connected (first proved in \cite{Ha}), every pair of vertices has finite distance, and the curve graph is a \textit{geodesic space}$ - $i.e., there is a path of minimal length (a \textit{geodesic}) connecting any two vertices.

\indent Although distance is defined in terms of intersections, there is not a strict relationship between intersection number and distance.  Broadly speaking, large distance implies large intersection number (specifically, $ d(\alpha, \beta) \leq 2log_2 [i(\alpha, \beta)]+2 $ (\cite{He})); curves with large intersection number will generally have large distance, but this need not be the case.  There are some facts that we can assert in simple cases, however.

\begin{fact}
If $ i(\alpha, \beta) $ is 1 or 2, then $ d(\alpha, \beta)=2 $ (if we take an annular neighborhood of $ \alpha \cup \beta $, the components of its boundary will be disjoint from both curves, and at least one of them will be essential).
\end{fact}

\begin{fact}
The pair $ (\alpha, \beta) $ is a \textit{filling pair} if no essential scc is disjoint from $ \alpha \cup \beta- $i.e., any essential scc must intersect at least one of the pair.  It's clear from the definition that $ d(\alpha, \beta)>2 $ if and only if $ (\alpha, \beta) $ is a filling pair.
\end{fact}

\begin{fact}
In a closed surface $ S_g $, $ (\alpha, \beta) $ is a filling pair if and only if $ S_g \smallsetminus (\alpha \cup \beta) $ is a union of disks.  These disks are shaped like polygons with alternating $ \alpha $-and $ \beta $-edges; the curves partition each other into arcs and, when we cut the surface apart into disks, we produce two copies of each arc.  If one of these disk-regions has two edges which are copies of the same arc then, by connecting them, we produce an essential scc which intersects one of the pair $ (\alpha, \beta) $ once and is disjoint from the other:  by Fact 2, it follows that $ d(\alpha, \beta)=3 $.
\end{fact}

\begin{fact}
It can be shown that, for a filling pair on $ S_g $, we have $i(\alpha, \beta) \geq 2g-1 $ and that, in the case of minimal intersection number, $ d(\alpha, \beta)=3 $ (see Proposition 3.1).
\end{fact}

\indent Let $ F_g $ denote the surface $ S_g $ together with a minimally-intersecting filling pair.  Let Mod($ S_g $) denote the \textit{mapping class group} of $ S_g $.  This is the group of homotopy \textit{classes} of orientation-preserving homeomorphisms (\textit{mappings}) on $ S_g $.  This group acts on $ S_g $ and we identify curves in the surface according to this action$ - $thus \textit{an} $ F_g $ is $ S_g $ with a particular filling pair, and a homeomorphically different filling pair on the same surface is a \textit{different} $ F_g $.  In their 2015 paper \cite{AH}, Aougab and Huang addressed the following questions:

\begin{enumerate}
\item For which genera $ g $ does an $ F_g $ exist?
\item How many $ F_g $s are there for a given $ g $?
\end{enumerate}

\indent They showed that an $ F_g $ exists for every genus $ g > 2 $, and then demonstrated upper and lower bounds on the number of distinct $ F_g $s.  Specifically, they proved the following:

\begin{theorem}
Let $ N(g) $ be the number of $ F_g $s (the number of Mod($ S_g $)-orbits of minimally-intersecting filling pairs on $ S_g $).  Then there exists a function $ f(g)=O(\dfrac{3^{{\frac{g}{2}}}}{g^2}) $ such that $ f(g) \leq N(g) \leq 2^{2g-2} (4g-5) (2g-3)! $ for all $ g>2 $.
\end{theorem}

\indent To prove these bounds, they introduced a means of representing the filling pair $ (\alpha, \beta) $ on $ F_g $ by a permutation $ \sigma \in \Sigma_{8g-4} $, where $ \Sigma_{8g-4} $ is the symmetric group on $ 8g-4 $ symbols (this $ \sigma $ is known as a \textit{filling permutation} of $ (\alpha, \beta) $).  It will be useful to note here that a defining feature of $ \sigma $ is that it satisfies the equation

\begin{equation}
\sigma Q^{4g-2} \sigma=\tau
\label{permeqn}
\end{equation}

\noindent where $ Q $ and $ \tau $ are permutations defined in Section 2 below.  They then demonstrated that the Mod($ S_g $)-orbit of $ (\alpha, \beta) $ corresponds to a special conjugacy class of $ \sigma $.  More precisely:

\begin{proposition}
Let $ (\alpha, \beta) $ and $ (\alpha', \beta') $ be minimally-intersecting filling pairs on $ S_g $, and let $ \sigma, \sigma' \in \Sigma_{8g-4} $ be filling permutations for the two pairs, respectively.  There is a subgroup $ T \leq \Sigma_{8g-4} $ such that $ (\alpha, \beta) $ and $ (\alpha', \beta') $ belong to the same Mod($ S_g $)-orbit if and only if $ \sigma $ and $ \sigma' $ are conjugate by an element of $ T $. 
\end{proposition}

\indent This $ T $ is independent of the pair $ (\alpha, \beta) $:  all homeomorphic filling pairs have filling permutations which are conjugate by an element of $ T $.  The upper bound, then, is simply the maximum number of such conjugacy classes:  the maximum number of possible filling permutations divided by the minimum possible order of $ T $.

\indent To prove the lower bound (as well as the mere existence of \textit{any} $ F_g $s), Aougab and Huang produced an explicit construction.  This construction involves producing a higher-genus $ F $ from one of lower genus by iteratively attaching ``$ Z $-pieces":  copies of a fixed surface $ Z $ of genus 2, together with a filling pair that intersect six times.  See Figure \ref{Zfig}.

\begin{figure}[htb]
\begin{center}
\includegraphics[scale=.2]{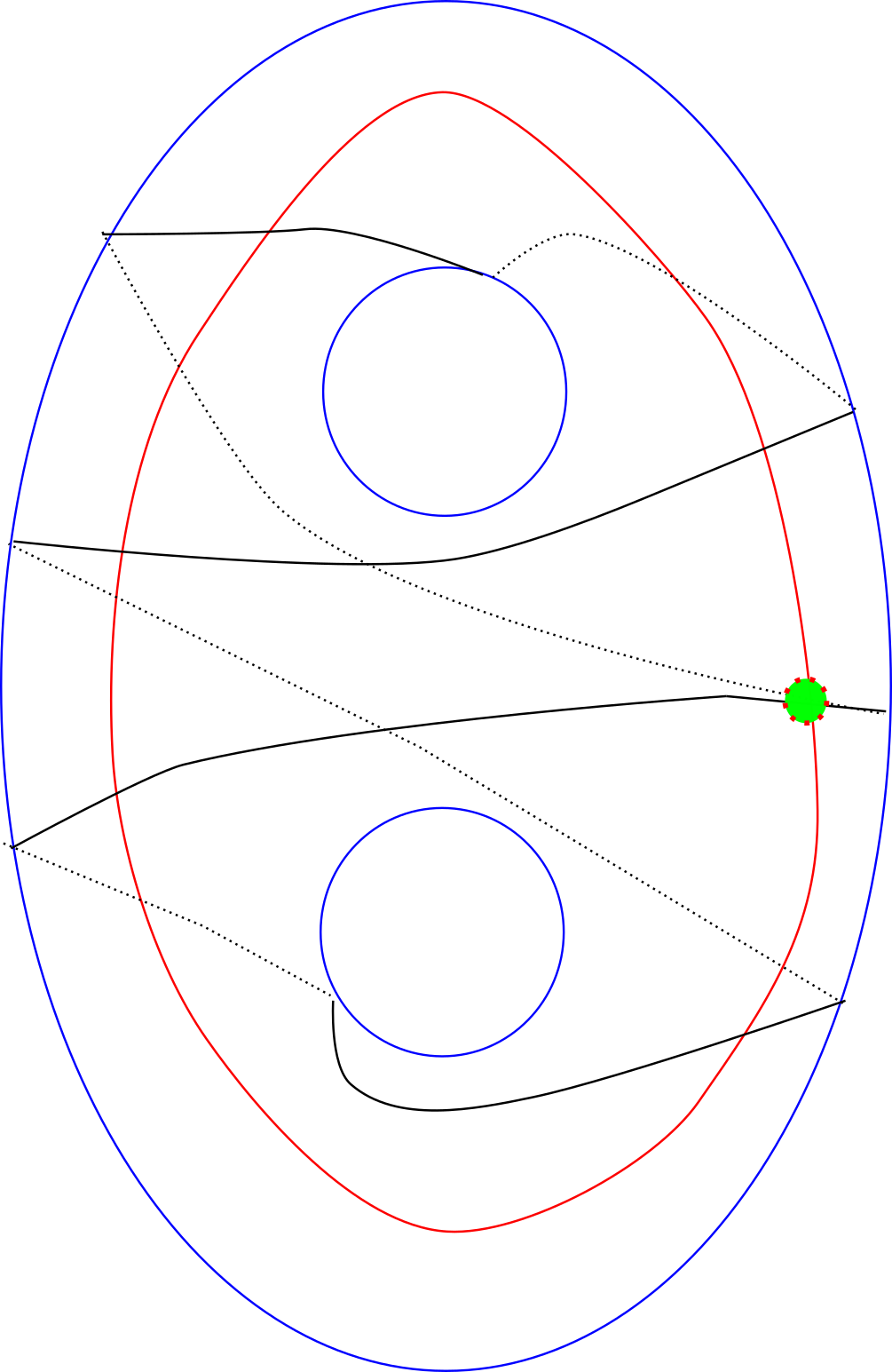}
\caption{$ Z $ with curves $ \alpha $ (shown in red) and $ \beta $ (shown in black).  The green vertex is adjacent to all four regions.\label{Zfig}}
\end{center}
\end{figure}

\indent A $ Z $-piece is attached to an $ F_g $ by connected sum:  delete a disk neighborhood from each surface, and then identify the boundary curves produced by this deletion.  This produces the surface $ S_{g+2} $.  Producing a minimally-intersecting filling pair on this new surface requires a particular choice of the deleted disks.  The filling pair on $ Z $ has a distinguished intersection which is adjacent to each of the (four) complementary regions of the filling pair (throughout this paper, we will refer to this intersection as the ``green vertex").  By deleting a disk neighborhood of the green vertex from $ Z $ and a disk neighborhood of \textit{any} intersection of $ F_g $, we create boundary in the curves:  they become arcs.  When we identify the boundaries of the surfaces (circles), we simultaneously identify the boundaries of these arcs (endpoints on the circles).  This produces a filling pair on $ S_{g+2} $ with exactly the minimum number of intersections on the new surface$ - $i.e., it produces an $ F_{g+2} $.  Figures \ref{cutsurfacesfig} and \ref{F3fig} illustrate the case $ g=1 $.

\begin{figure}[htb]
\begin{center}
\includegraphics[scale=.18]{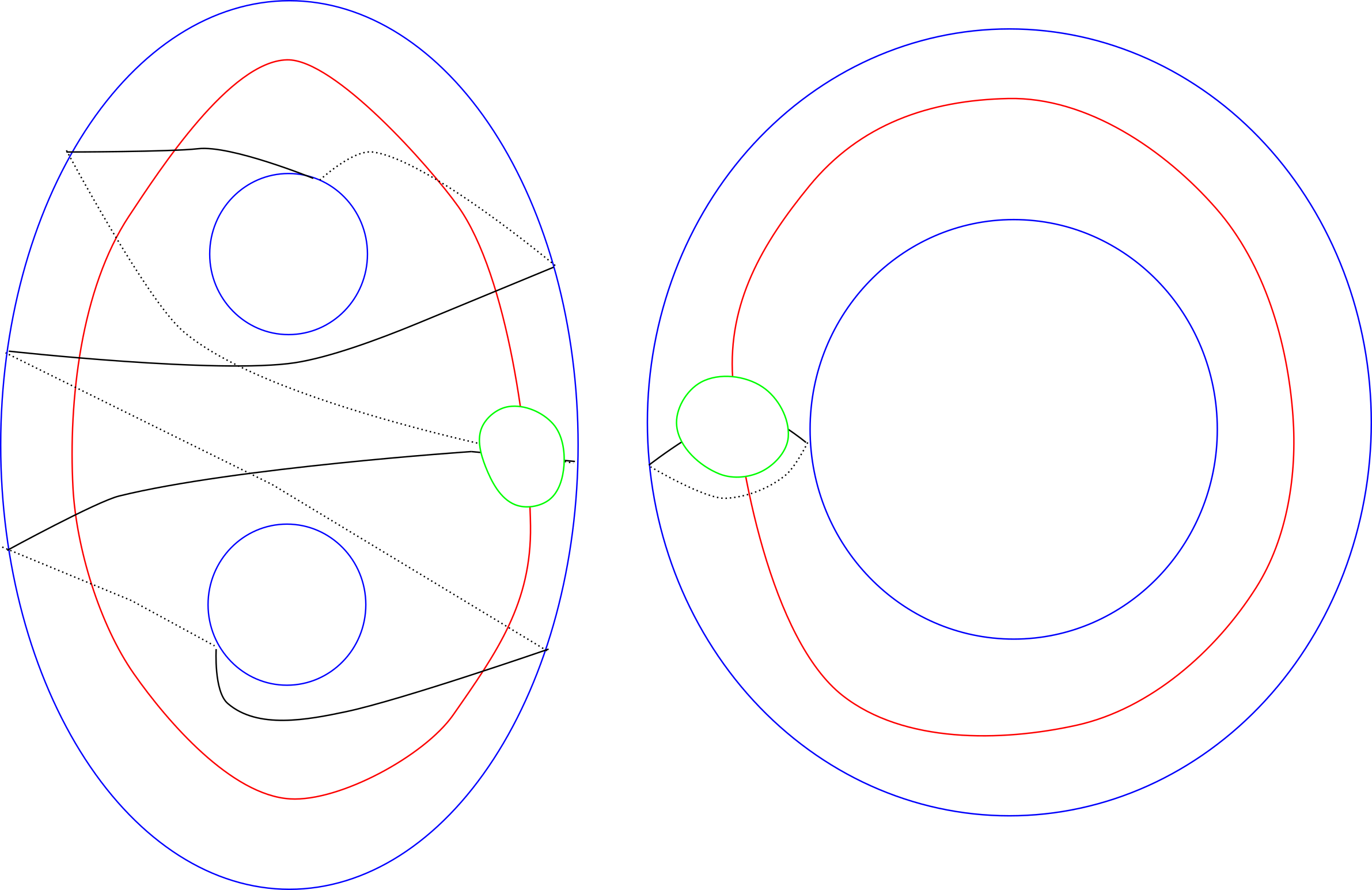}
\caption{$ Z $ and $ F_1 $ with appropriate disks removed.  The curves $ \alpha $ and $ \beta $ on $ Z $ and $ F_1 $ are now arcs.\label{cutsurfacesfig}}
\end{center}
\end{figure}

\begin{figure}[htb]
\begin{center}
\includegraphics[scale=.19]{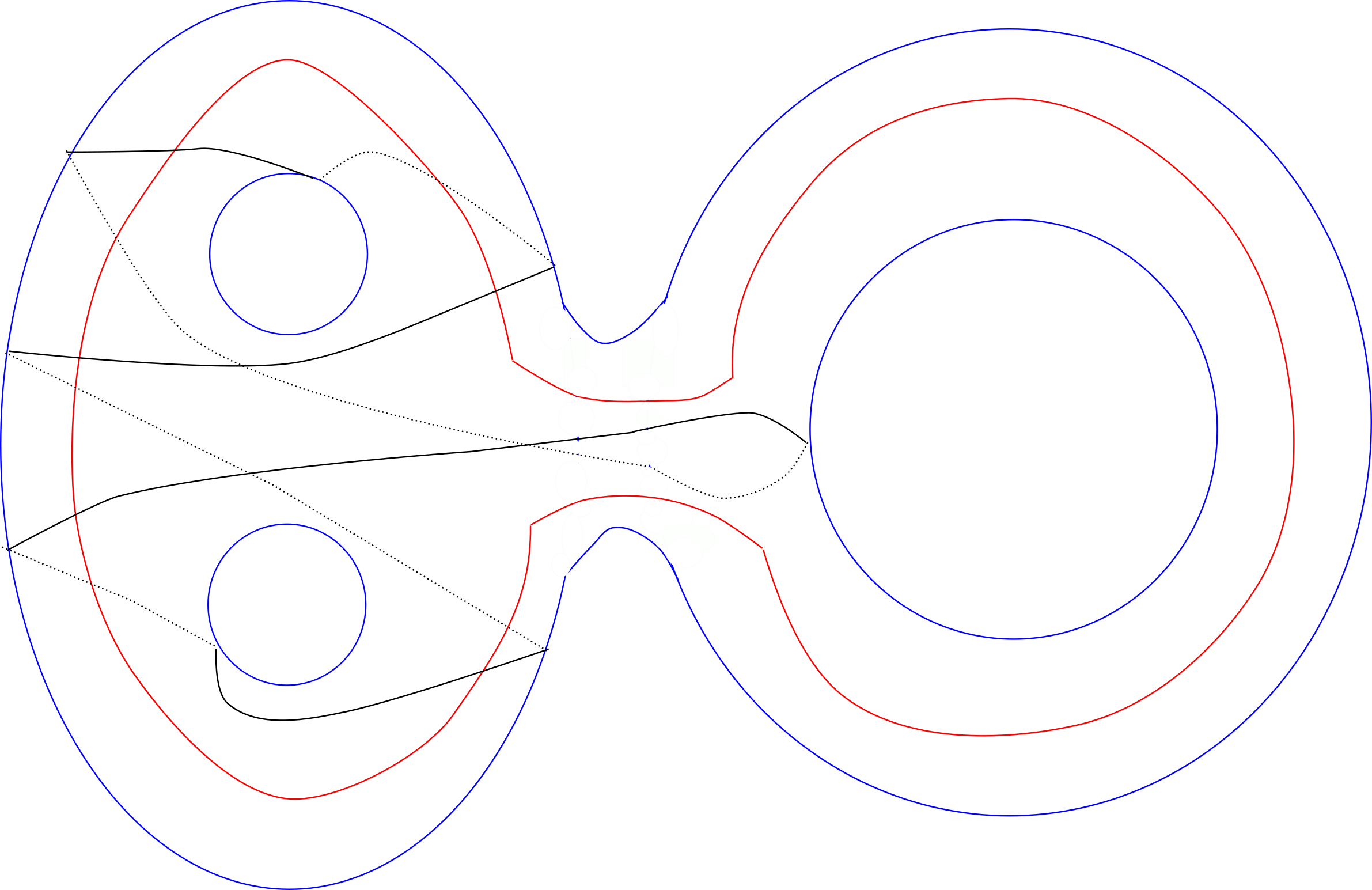}
\caption{Arcs of the appropriate color on each surface have had their endpoints identified, producing a new pair of scc's on a new closed surface.\label{F3fig}}
\end{center}
\end{figure}

\indent This splits the existence problem into two cases:  odd genus and even genus.  To show, for arbitrary $ g $, that an $ F_g $ exists, it suffices to show that an $ F_1 $ exists (if $ g $ is odd) or an $ F_2 $ exists (if $ g $ is even); it follows by induction that an $ F_g $ exists.  There is an obvious candidate on $ S_1- $namely, the longitudinal and meridional curves of the torus, which intersect $ 2(1)-1=1 $ time$ - $and so existence follows for all odd genera.  It turns out that there is no $ F_2- $if $ (\alpha, \beta) $ is a filling pair on $ S_2 $, then $ i(\alpha, \beta) \geq 4>3=2(2)-1- $however, this is the only defect in the result:  there exists an $ F_4 $ and, therefore, existence follows for all even genera $ g \geq 4 $.

\indent Because the construction allows us to select any intersection on $ F_g $ at which to attach a $ Z $-piece, there are (usually) many different ways to produce a filling pair on the same surface $ S_g $; moreover, each time we attach a $ Z $-piece, we add four new intersections, which means we increase the number of choices of attachment point for each successive $ Z $-piece.  In light of this observation, it makes sense intuitively that the lower bound will be exponential in $ g $.  We obtain it (up to asymptotic equivalence) by taking the \textit{minimum} number of filling permutations that can be constructed by successive additions of $ Z $-pieces divided by the \textit{maximum} order of $ T $.

\indent It is a vital factor in counting the minimum number of possible filling permutations that $ Z $-pieces can always be removed \textit{cleanly}:  given an $ F_g $ with $ g>3 $, if it is possible to decompose $ F_g $ in more than one way into a connected sum$ - $say, $ F_{g-2}\ \sharp\ Z $ and $ F'_{g-2}\ \sharp\ Z' $, with $ Z \neq Z'- $then $ Z $ and $ Z' $ can overlap only on the boundary curves where each piece is attached to the larger surface:  the interiors of the two surfaces must be disjoint (note that the ``remainder" surfaces $ F_{g-2} $ and $ F'_{g-2} $ may be different).  By successively removing $ Z $-pieces, we could (in principle, at least) decompose an $ F_g $ uniquely into an irreducible ``kernel" surface with $ Z $-pieces attached at specified points.

\indent This hints at a more general phenomenon, and prompts us to ask a number of questions:

\begin{enumerate}
\item Is $ Z $ the only piece?  That is, does there exist an $ F_g $ that decomposes as $ F_{g-2}\ \sharp\ S_2 $, where the filling pairs on $ S_2 $ and $ Z $ are \textit{not} homeomorphic?
\item Are there higher-genus pieces?  Is it possible to perform constructions or obtain decompositions of surfaces of the form

\begin{equation}
F_{g-k}\ \sharp\ Z_k
\label{decompeqn}
\end{equation}
where $ Z_k $ has genus $ k $ with $ k>2 $?
\item Is there a criterion that will determine when a surface admits a decomposition like that in (\ref{decompeqn}), and that will indicate how to perform the decomposition?
\end{enumerate}

\indent Each of these questions is answered below in the affirmative.  Perhaps unsurprisingly, the decomposition criterion depends on filling permutations.  Specifically:

\begin{theorem}
Let $ F_g $ be a surface of genus $ g>2 $ together with a filling pair $ (\alpha, \beta) $ which intersect $ 2g-1 $ times, and let $ \sigma \in \Sigma_{8g-4} $ be a filling permutation for $ (\alpha, \beta) $.  Then $ F_g $ admits a decomposition of the form (\ref{decompeqn}) if and only if there exist ``appropriately separated" numbers $ x,a,y,b \in \lbrace 1,2, \ldots, 8g-4 \rbrace $ which satisfy a list of six permutation equations involving $ \sigma, Q^{4g-2} $, and $ \tau $.
\end{theorem}

\indent See Theorem 5.1 below for the full statement of the result.

\indent As we have said (see Fact 4), if $ (\alpha, \beta) $ is the filling pair on an $ F_g $, then $ d(\alpha, \beta)=3 $.  By extending $ F_g $ into an $ F_{g+2} $, we map this pair of vertices in $ \mathcal{C}(F_g) $ to a pair $ (\iota_g(\alpha), \iota_g(\beta)) $ in $ \mathcal{C}(F_{g+2}) $ with the same distance (since the new pair intersects minimally on the new surface).  But we achieved this extension by performing a surgery in a disk neighborhood of one point of $ \alpha\ \cap\ \beta $; since we can isotope any other scc in $ F_g $ away from this neighborhood, it follows that no other curve is affected by the extension.  Thus the map $ \iota_g: \mathcal{C}^1(F_g) \rightarrow \mathcal{C}^1(F_{g+2}) $ acts as inclusion on every other curve in $ F_g $; we therefore have an embedding of one curve graph into the other which does not increase distance:  all intersection numbers among curves in $ F_g $ other than $ \alpha $ and $ \beta $ are preserved (including intersections of either $ \alpha $ or $ \beta $ with any \textit{other} curve), and so any path in $ \mathcal{C}^1(F_g) $ connecting two such vertices is embedded in $ \mathcal{C}^1(F_{g+2}) $.  We have a similar embedding $ \iota_2: \mathcal{C}^1(Z) \rightarrow \mathcal{C}^1(F_{g+2}) $.

\indent A surface $ Y $ is an \textit{essential subsurface} of $ S_g $ if $ Y \subset S_g $ and every boundary component of $ Y $ is an essential curve in $ S_g $.  If $ Y $ is not an annulus, we have the following map on 0-skeleta, $ \pi_Y: \mathcal{C}^0(S_g) \rightarrow \mathcal{P}(\mathcal{AC}^0(Y)) $ (where $ \mathcal{P} $ denotes the power set):  for a vertex in $ \mathcal{C}(S_g) $, take a representative $ \gamma $ of its class that intersects $ Y $ minimally; then $ \pi_Y(\gamma)=\gamma \cap Y $.  The image of $ \gamma $ is a (possibly empty) union of disjoint curves and/or arcs$ - $the segments of $ \gamma $ where it cuts through the subsurface.  This map $ \pi_Y $ is the \textit{subsurface projection}.

\indent $ F_{g,1} $ (the surface with one of its vertices removed) can be thought of as an essential subsurface of $ F_{g+2} $; for every curve $ \gamma $ in $ F_g $ other than $ \alpha $ and $ \beta $, $ \iota_g(\gamma) \cap F_{g,1}=\iota_g(\gamma) $ (since $ \iota_g(\gamma) $ is completely contained in $ F_{g,1} $), and so the subsurface projection sends $ \iota_g(\gamma) $ to $ \gamma $:  $ \pi_{F_{g,1}} \circ \iota_g $ is the identity on all vertices of $ \mathcal{C}(F_g) $ except $ \alpha $ and $ \beta $ (and similarly for $ \pi_Y $ and $ \iota_2 $, where $ Y $ is $ Z $ with the green vertex removed, viewed as a subsurface of $ F_{g+2} $).  The maps $ \iota_g $ and $ \iota_2 $ are (almost) one-sided inverses of the subsurface projection.  In diagram form:

\begin{Large}
\begin{displaymath}
\xymatrix{ &\mathcal{C}^1(F_{g+2}) \ar@/_/[dl]_{\pi_{F_{g,1}}} \ar@/^/[dr]^{\pi_Y} &\\
\mathcal{C}^1(F_g) \ar@/_/@{^{(}->}[ur]_{\iota_g} &  &\mathcal{C}^1(Z) \ar@/^/@{_{(}->}[ul]^{\iota_2}}
\end{displaymath}
\end{Large}

\indent Theorem 1.2 generalizes this picture:

\begin{Large}
\begin{displaymath}
\xymatrix{ &\mathcal{C}^1(F_{l+k}) \ar@/_/[dl]_{\pi_{F_{l,1}}} \ar@/^/[dr]^{\pi_{Z_{k,1}}} &\\
\mathcal{C}^1(F_l) \ar@/_/@{^{(}->}[ur]_{\iota_l} &  &\mathcal{C}^1(Z_k) \ar@/^/@{_{(}->}[ul]^{\iota_k}}
\end{displaymath}
\end{Large}

The maps $ \iota_l: \mathcal{C}^1(F_l) \rightarrow \mathcal{C}^1(F_{l+k}) $ and $ \iota_k: \mathcal{C}^1(Z_k) \rightarrow \mathcal{C}^1(F_{l+k}) $ are almost one-sided inverses to the appropriate subsurface projections.

\section{Filling Pairs and Filling Permutations}

Let $ (\alpha, \beta) $ be a filling pair on $ S_g $.  As we've said, $ S_g \smallsetminus (\alpha \cup \beta) $ is a union of polygons; if $ i(\alpha, \beta)=n $, then the sum of the edges of these polygons is $ 4n $ (each curve has $ n $ arcs and, when we cut open the surface, we produce two copies of each arc).  Let $ \Sigma_{4n} $ be the symmetric group on $ \lbrace 1,2, \ldots, 4n \rbrace $.  We will define a filling permutation $ \sigma \in \Sigma_{4n} $ for $ (\alpha, \beta) $, which will record the information about how the two curves intersect each other, and is sufficient to completely determine the surface and its filling pair.  This definition and the associated results are generalizations of those introduced in \cite{AH}.  We define $ \sigma $ as follows:

\indent Orient $ \alpha $ and $ \beta $ and choose an initial arc of each.  Label these arcs $ \alpha_1 $ and $ \beta_1 $ (respectively), and then label the other arcs of each curve in the order of orientation.  The union of polygons will then have edges with labels in $ \lbrace \alpha_i, \beta_i, \overline{\alpha_i}, \overline{\beta_i}:  1 \leq i \leq n \rbrace $.  We will assign to these labels elements of $ \lbrace 1,2, \ldots, 4n \rbrace $ by the convention:

$ \alpha_i \mapsto 2i-1 $

$ \beta_i \mapsto 2i $

$ \overline{\alpha_i} \mapsto 2i-1+2n $

$ \overline{\beta_i} \mapsto 2i+2n $.

We will often equivocate between ``an edge," ``the label of an edge," and ``the number assigned to the label of an edge."

\indent For each region of $ S_g \smallsetminus (\alpha \cup \beta) $, construct a cycle whose entries are the numbers of the edge labels, in clockwise order (note that, as a permutation, the cycle is independent of the choice of starting edge).  Let $ \sigma $ be the product of these (commuting) cycles.  Then $ \sigma $ is the \textit{filling permutation} of the (oriented) pair $ (\alpha, \beta) $.

\begin{example}
Consider the surface $ Z $ introduced above (see Figure \ref{Zfig}).  $ Z $ has a filling pair $ (\alpha, \beta) $ with intersection number 6,  $ Z \smallsetminus (\alpha \cup \beta) $ has four regions (two octagons and two rectangles), and the green vertex (as well as one other vertex) is adjacent to all four regions.  Orient the curves as shown in Figure \ref{Zorientfig} and choose initial arcs to be the ones whose origin is the green vertex.  Then $ \sigma=(1,10,15,20,17,22,3,12)(24,5,18,11)(23,16,9,6,7,4,21,14) $

$ (2,19,8,13) $.  \indent \indent \indent \indent \indent \indent \indent \indent \indent \indent \indent \indent \indent \indent \indent \indent \ \ \ \ $ \square $

\begin{figure}[ht]
\begin{center}
\includegraphics[scale=.2]{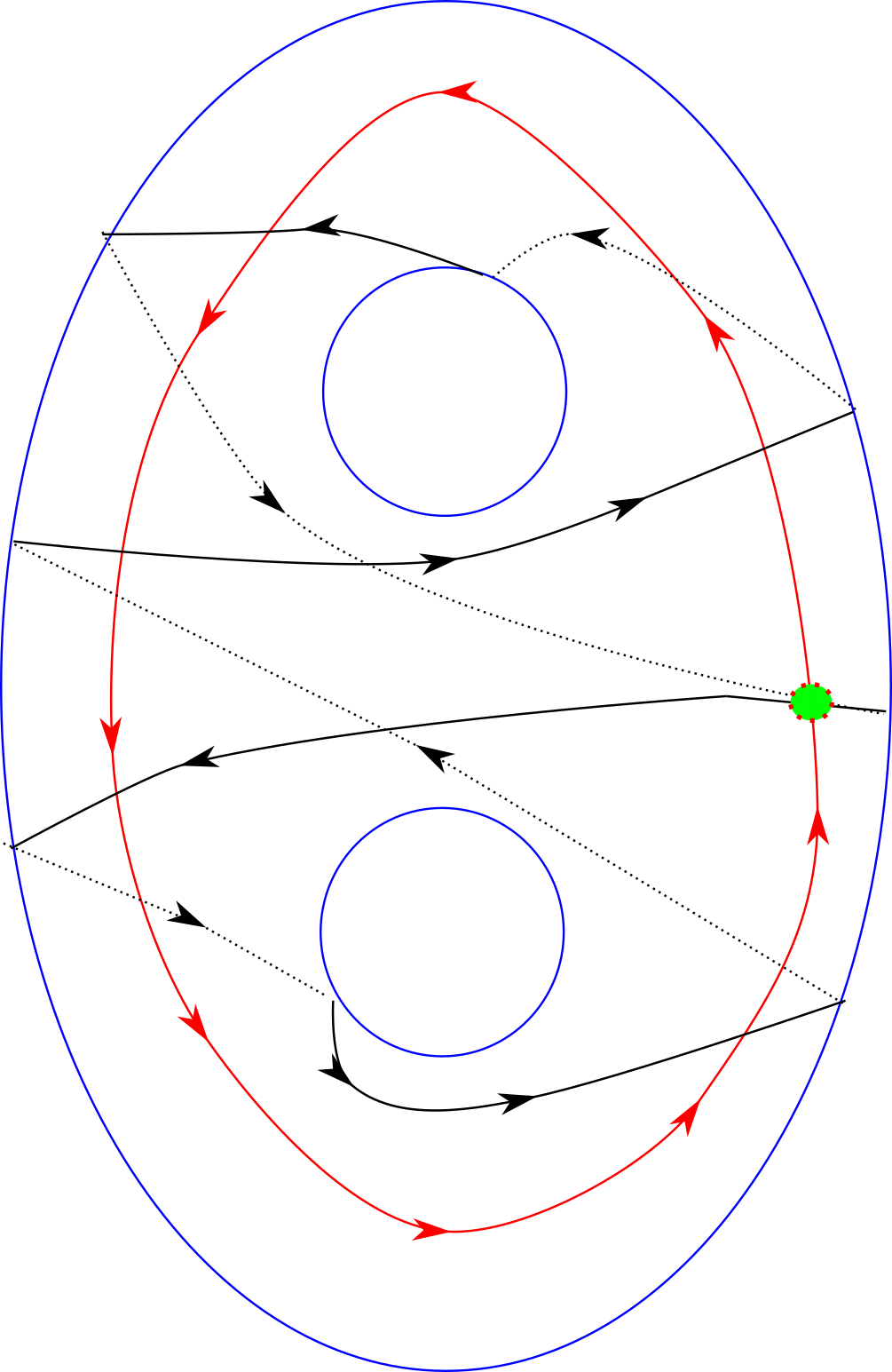}
\caption{$ Z $ with the curves $ \alpha $ and $ \beta $ now oriented.\label{Zorientfig}}
\end{center}
\end{figure}
\end{example}

\indent In \cite{AH}, the definition is given for minimally-intersecting filling pairs only; as we will see (Proposition 3.1), the filling pair $ (\alpha, \beta) $ intersects minimally if and only if $ S_g \smallsetminus (\alpha \cup \beta) $ consists of only one disk-region.  Thus, filling permutations in \cite{AH} are always a single cycle, while our definition allows them to have more than one cyclic factor.  Apart from some slight streamlining, the proofs of the results below mainly differ from the analogous results in \cite{AH} in that they allow for $ S_g \smallsetminus (\alpha \cup \beta) $ to have more than one component, and for $ \sigma $ to consist of more than one cycle.

\begin{lemma}
Let $ (\alpha, \beta) $ be a filling pair on a surface $ S_g $ with $ i(\alpha, \beta)=n $, and let $ \sigma \in \Sigma_{4n} $ be a filling permutation for $ (\alpha, \beta) $.  Then $ \sigma $ determines $ g- $specifically:
\begin{equation}
g=1+\frac{1}{2} (n-c),
\label{genuseq}
\end{equation}
where $ c $ is the number of cyclic factors of $ \sigma $.
\end{lemma}

\begin{proof}
$ \alpha \cup \beta $ is an embedded graph in $ S_g $:  its vertices are the elements of $ \alpha\ \cap\ \beta $, its edges are the components of $ \alpha \bigtriangleup \beta $, and its regions are the components of $ S_g \smallsetminus (\alpha \cup \beta) $.  Since $ (\alpha, \beta) $ is a filling pair, all of these regions are disks, and therefore this graph gives a cell decomposition for $ S_g $.  Thus, we can compute its Euler characteristic by $ \chi(S_g)=V-E+F $.

$ V=i(\alpha, \beta)=n $, $ F=c $, and since the graph is 4-valent, it follows that $ E=2V $.  So:

$ \chi(S_g)=n-2n+c $

$ \chi(S_g)=-n+c $.  Since $ S_g $ is a closed surface, we have

$ -n+c=\chi(S_g)=2-2g $.  Solving for $ g $, we get (\ref{genuseq}).
\end{proof}

\indent Since the surface on which $ (\alpha, \beta) $ is a filling pair is determined by its filling permutation, this suggests that both the surface and the pair can be defined by the filling permutation.  All we require is a way to characterize filling permutations in purely algebraic terms.  Let $ Q, \tau \in \Sigma_{4n} $ be the permutations $ Q=(1,2, \ldots, 4n) $ and $ \tau=(1,3, \ldots, 2n-1)(2,4, \ldots, 2n)(4n-1,4n-3, \ldots, 2n+1)(4n,4n-2, \ldots, 2n+2) $.  Let $ N=2n $.  Then we have the following result.

\begin{theorem}
Let $ \sigma \in \Sigma_{4n} $ be a filling permutation.  Then $ \sigma $ is a product of cycles which alternate odd and even entries, and it satisfies the permutation equation
\begin{equation}
\sigma Q^N \sigma=\tau.
\label{Npermeq}
\end{equation}
\indent Conversely, if $ \sigma $ is a permutation whose cyclic factors alternate odd and even entries and satisfies (\ref{Npermeq}), then $ \sigma $ is the filling permutation of a filling pair $ (\alpha, \beta) $ on some surface $ S_g $.
\end{theorem}

\indent Compare Lemma 2.2 in \cite{AH}.

\begin{proof}
As we observed in Fact 3, the polygonal components of $ S_g \smallsetminus (\alpha \cup \beta) $ have alternating $ \alpha $-and $ \beta $-edges; since $ \alpha $-edges have odd numbers and $ \beta $-edges have even ones, the cyclic factors of $ \sigma $ must alternate odd and even entries.

\indent To see that $ \sigma $ satisfies (\ref{Npermeq}), consider a crossing of $ \alpha $ and $ \beta $ with arcs $ \alpha_i, \alpha_{i+1}, \beta_j, \beta_{j+1} $ (where $ n+1=1 $).  Cutting open the surface along $ \alpha \cup \beta $ splits the crossing into four corners (see Figure \ref{cornerfig}).

\begin{figure}[htb]
\begin{center}
\includegraphics[scale=.3]{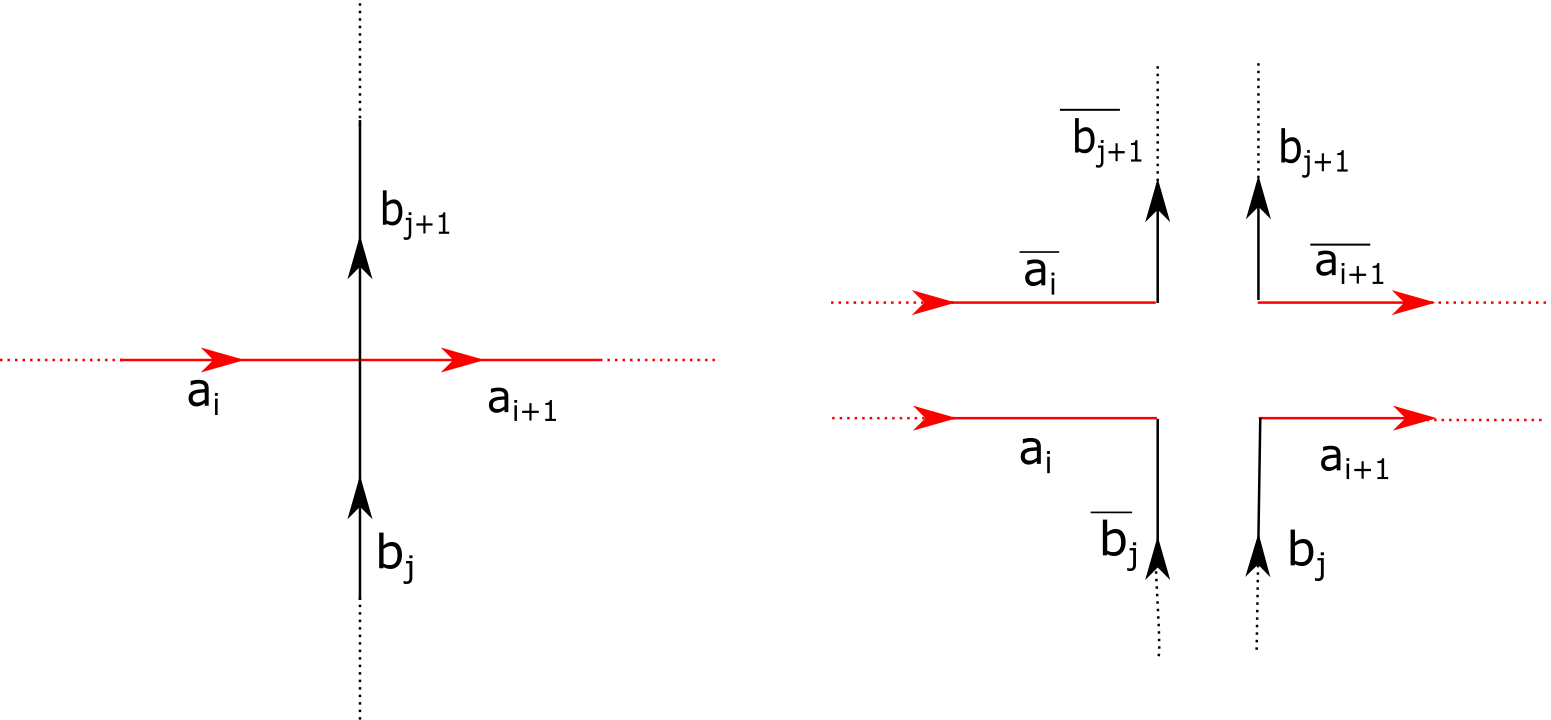}
\caption{A crossing of $ \alpha \cup \beta $ split into four corners on the union of polygons.\label{cornerfig}}
\end{center}
\end{figure}

\indent The permutations $ \sigma, Q^N $, and $ \tau $ act on the set of edges of the union of polygons:  $ \sigma $ sends each edge to the edge next to it in the clockwise direction, $ Q^N $ sends each edge to its opposite, and $ \tau $ cycles the edges of $ \alpha $ and $ \beta $ in the order of their labeling (cycling positive edges in the forward order and negative edges in the reverse order).  Applying the left side of (\ref{Npermeq}) to $ \alpha_i $, we see that $ \sigma $ sends it to $ \overline{\beta_j} $, $ Q^N $ sends $ \overline{\beta_j} $ to $ \beta_j $, and $ \sigma $ sends $ \beta_j $ to $ \alpha_{i+1}=\tau(\alpha_i) $, in agreement with the right side of (\ref{Npermeq}).  Similarly, $ \sigma Q^N \sigma(\beta_j)=\tau(\beta_j) $.  Applying the left side to $ \overline{\alpha_{i+1}} $ sends it to $ \beta_{j+1} $, which then goes to $ \overline{\beta_{j+1}} $, which then goes to $ \overline{\alpha_i}=\tau(\overline{\alpha_{i+1}}) $ (and similarly for $ \overline{\beta_{j+1}}) $.  Since $ i $ and $ j $ are arbitrary, (\ref{Npermeq}) holds.

\indent Conversely, suppose $ \sigma $ is as described in the statement of the theorem.  We will construct a surface $ S_g $ with a filling pair $ (\alpha, \beta) $ for which $ \sigma $ can be a filling permutation.  For each cyclic factor of $ \sigma $, take a polygon with as many sides as the cycle has entries.  Assign to the edges of the polygon the labels in $ \lbrace \alpha_i, \beta_i, \overline{\alpha_i}, \overline{\beta_i}:  1 \leq i \leq n \rbrace $ that correspond to the entries of the cycle, going around the polygon in the clockwise direction.  We show that, when these polygons are glued together by identifying edges with opposite labels, we obtain a closed, orientable surface, and the edges form a filling pair of scc's.

\indent $ \sigma, Q^N $, and $ \tau $ act on the edges of the polygons as stated above.  Take the corner with $ \alpha_i $ on the left side (that is, counterclockwise to the vertex.  In Figure \ref{cornerfig}, this is the lower left corner).  The left side of (\ref{Npermeq}) sends $ \alpha_i $ to the edge that will follow it once opposite $ \beta $-edges are identified; according to (\ref{Npermeq}), this edge must be $ \tau(\alpha_i)=\alpha_{i+1} $.  Since $ i $ is arbitrary, each $ \alpha $-edge is followed by the one with the next label (with $ \alpha_n $ followed by $ \alpha_1 $); since $ \tau $ cycles all the positive $ \alpha $-edges, the resulting curve must contain all of them:  $ \alpha $ is a single, closed curve.  Similarly, all the negative $ \alpha $-edges are looped in the order of their labeling, and all the $ \beta $-edges must form a single curve when opposite $ \alpha $-edges are identified.

\indent We now show that the $ 4n $ vertices of the union of polygons are identified in fours, so that $ i(\alpha, \beta)=n $.  The permutation $ Q^N \sigma $ sends the left edge of a vertex to the left edge of a vertex identified with it, so it suffices to show that $ (Q^N \sigma)^4=1 $.

$ Q^N \sigma Q^N \sigma Q^N \sigma Q^N \sigma=Q^N  (\sigma Q^N \sigma) Q^N (\sigma Q^N \sigma)=Q^N \tau Q^N \tau $.

\indent This product acts on a positive edge by sending it forward by 1 ($ \tau $), sending the next edge to its opposite ($ Q^N $), sending the opposite edge \textit{backward} by 1 (since $ \tau $ sends positive and negative edges in opposite directions), and then sending this edge to its opposite$ - $which is the edge with which we started.  Negative edges follow a similar circuit.  Thus $ (Q^N \tau)^2 $ fixes every edge, and is therefore the identity permutation.  Notice that we have just proved $ \tau $ and $ Q^N $ anticommute.  This will be a useful fact in what follows.

\indent Gluing together opposite edges of the polygons produces a surface.  Since all the boundary edges have been identified in pairs, the surface is closed; since we have identified opposite-oriented edges, the surface is orientable.  $ (\alpha, \beta) $ is obviously a filling pair, since their complement is precisely the union of polygons with which we started.  By Lemma 2.1, then, the surface produced by the gluing is $ S_g $, where $ g $ is given by (\ref{genuseq}).
\end{proof}

\indent The construction of the filling permutation from the pair $ (\alpha, \beta) $ depended on the introduction of additional topological and combinatorial structure$ - $ namely, the orientation of the two curves and the labeling of their arcs.  If we made different choices for this information, we would obtain a different filling permutation.  We will refer to $ (\alpha, \beta) $ together with orientations and labelings as an \textit{oriented filling pair}, and the pair without this extra information as an \textit{ordinary filling pair}.  We would like to know the relationship between two filling permutations whose oriented filling pairs have the same underlying ordinary filling pair.  Since we will be interested in identifying pairs up to the action of Mod($ S_g $)), we would also like to know how homeomorphisms affect filling permutations.

\begin{proposition}
Let $ \kappa, \delta, \eta, \mu \in \Sigma_{4n} $ be the permutations:

$ \kappa=(1,3, \ldots, 2n-1)(2n+1,2n+3, \ldots, 4n-1) $

$ \delta=(2,4, \ldots, 2n)(2n+2,2n+4, \ldots, 4n) $

$ \eta=(1,2n+1)(3,2n+3) \cdots (2n-1,4n-1) $

$ \mu=(1,2)(3,4) \cdots (4n-1,4n) $.  Let $ T=\langle \kappa, \delta, \eta, \mu \rangle \leq \Sigma_{4n} $.  If $ (\alpha, \beta) $ and $ (\alpha', \beta') $ are in the same Mod($ S_g $)-orbit, then their filling permutations are conjugate by an element of $ T $.
\end{proposition}

\indent Compare Lemma 2.3 in \cite{AH}.

\begin{proof}
If we restrict our attention to homeomorphic \textit{oriented} filling pairs, then there is no difference in the filling permutations.  let $ \sigma, \sigma' $ be the filling permutations for the two pairs (respectively).  The action of a homeomorphism on $ S_g $ which maps $ (\alpha, \beta) $ to $ (\alpha', \beta') $ gives an action mapping $ S_g \smallsetminus (\alpha \cup \beta) $ to $ S_g \smallsetminus (\alpha' \cup \beta') $.  This action must be simplicial (mapping vertices to vertices and edges to edges), and it must map polygons with the same number of sides to one another.

\indent Suppose there is a unique $ k $-gon among the components of $ S_g \smallsetminus (\alpha \cup \beta) $.  Then the homeomorphism maps it simplicially to itself$ - $i.e., the restriction of the homeomorphism to this $ k $-gon is an element of the dihedral group $ D_{2k} $; since it preserves orientation, it must be a rotation.  Rotating the $ k $-gon has the effect of cycling the entries of the factor of $ \sigma $ corresponding to the $ k $-gon$ - $but this does not change the underlying permutation.  Now suppose there are multiple $ k $-gons.  Each one is mapped simplicially onto another.  But this is also a rigid motion of the $ k $-gon (an element of $ D_{2k} $) and, by the same argument, it must be a rotation.  Thus, under the homeomorphism, each $ k $-gon in $ S_g \smallsetminus (\alpha \cup \beta) $ is replaced with a rotation of another $ k $-gon (possibly itself).  This has the effect of permuting the $ k $-cycle factors of $ \sigma $ and cycling their entries.  We have already observed that cycling the entries doesn't change the permutation and, since all the cyclic factors of a filling permutation commute, $ \sigma=\sigma' $.

\indent It remains to show how to obtain one filling permutation from another when the oriented filling pairs of each have the same ordinary filling pair.  This is equivalent to the question of how to obtain one oriented filling pair from another, and is simply a matter of relabeling:  we can interchange the roles of $ \alpha $ and $ \beta $ (conjugation by $ \mu $), change the orientation of $ \alpha $ (conjugation by $ \eta $) or $ \beta $ (a combination of the previous two), cycle the labels of $ \alpha $ (conjugation by a power of $ \kappa $), or cycle the labels of $ \beta $ (conjugation by a power of $ \delta $).  The elements of $ T $ (called \textit{twisting permutations}) give every possible relabeling action of the oriented filling pairs, thus completing the proof.
\end{proof}

\section{Minimally-Intersecting Filling Pairs}

Having introduced filling permutations as a combinatorial device for representing filling pairs, we will now restrict our attention to \textit{minimally-intersecting} filling pairs.  Our first task will be to determine this minimum.

\begin{proposition}
Let $ (\alpha, \beta) $ be a filling pair on $ S_g $.  Then $ i(\alpha, \beta) \geq 2g-1 $.
\end{proposition}

Compare Lemma 2.1 in \cite{AH}.

\begin{proof}
By the proof of Lemma 2.1, we have:

$ \chi(S_g)=-i(\alpha, \beta)+F $

$ 2-2g=-i(\alpha, \beta)+F $

$ i(\alpha, \beta)=2g-2+F \geq 2g-1 $.
\end{proof}

\indent Note that $ \alpha $ and $ \beta $ intersect minimally if and only if $ S_g \smallsetminus (\alpha \cup \beta) $ is a single disk (an $ (8g-4) $-gon).  By Fact 3, this implies that $ d(\alpha, \beta)=3 $ in the curve graph (this was the content of Fact 4).

\indent As we said in Section 1, we use the notation $ F_g $ to refer to the surface $ S_g $ together with a minimally-intersecting filling pair (up to homeomorphism).  We now show that the construction we described there realizes the minimum intersection number for nearly all genera.  Most of the work is done by the following lemma.

\begin{lemma}
Let $ Z $ and $ F_g $ be as described, and let $ F_g\ \sharp\ Z $ be the connected sum obtained by deleting a disk neighborhood of the green vertex from $ Z $, deleting a disk neighborhood of any vertex of $ F_g $, and identifying the boundaries of both the deleted disks and the pairs of arcs produced by deleting the disks.  Then $ F_g\ \sharp\ Z $ is an $ F_{g+2} $.
\end{lemma}

Compare "Proof of the Lower Bound" in \cite{AH}.

\begin{proof}
The key to the construction is that, when the two surfaces are glued together, the complement of the resulting curves is a single disk.  This implies that the new curves form a filling pair, and it immediately follows that they intersect minimally.  As we remarked, $ F_g \smallsetminus (\alpha \cup \beta) $ is an $ (8g-4) $-gon and, when we glue the surface together, the vertices are identified in fours; thus, deleting a disk neighborhood of one of $ F_g $'s vertices ``clips" four angles off of the $ (8g-4) $-gon.

\begin{figure}[ht]
\begin{center}
\includegraphics[scale=.23]{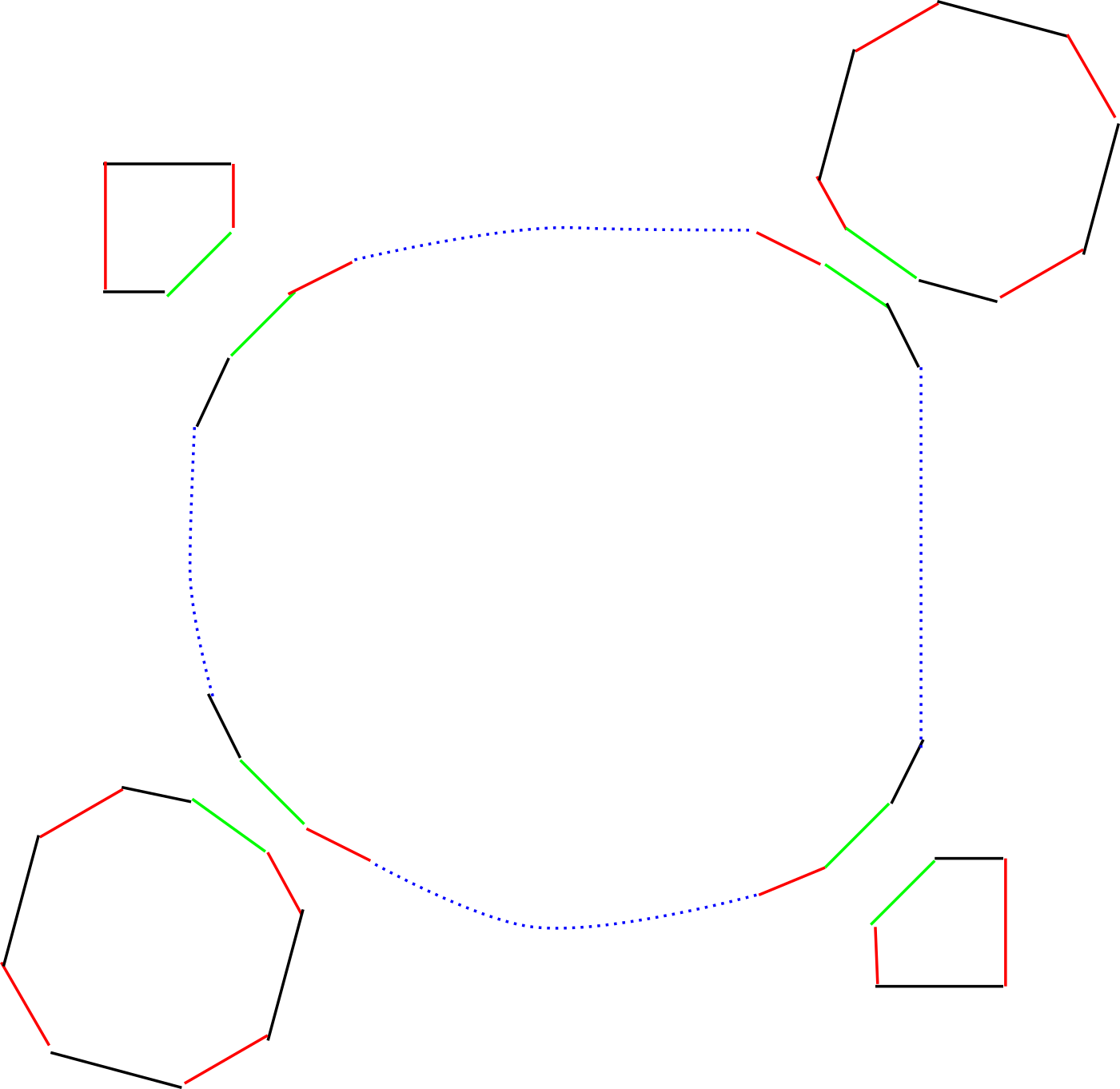}
\caption{$ F_g $ and $ Z $ represented as unions of disks.  The clipped edges (green) correspond to the intersections cut out of their respective surfaces.\label{clippeddisksfig}}
\end{center}
\end{figure}

\indent Similarly, when we delete the disk neighborhood of the green vertex, we clip one angle off each of the four complementary regions of $ Z $.  When we connect the two surfaces together, we glue the regions of $ Z $ onto the $ (8g-4) $-gon along these clipped edges$ - $gluing five disks together to produce one big disk.  This completes the proof.
\end{proof}

\indent With this lemma, all that's necessary to prove that an $ F_g $ exists for each $ g $ are two basal examples$ - $one of odd genus and one of even genus.  While the obvious genus-1 example (shown in Figure \ref{cutsurfacesfig}) suffices for odd genera, the situation is not quite as nice for even genera.  Before proving the existence result, we show the following \textit{nonexistence} result.

\begin{proposition}
An $ F_2 $ does not exist$ - $that is, for a filling pair $ (\alpha, \beta) $ on $ S_2 $, $ i(\alpha, \beta) \geq 4 $.
\end{proposition}

Compare Theorem 2.16 in \cite{AH}.

\begin{proof}
We attempt to construct a filling pair on $ S_2 $ with intersection number 3.  We represent $ S_2 $ as a rectangle with two holes in it, identifying the opposite sides of the rectangle and the boundaries of the two holes (see Figure \ref{nonF2fig}).

\begin{figure}[ht]
\begin{center}
\includegraphics[scale=.25]{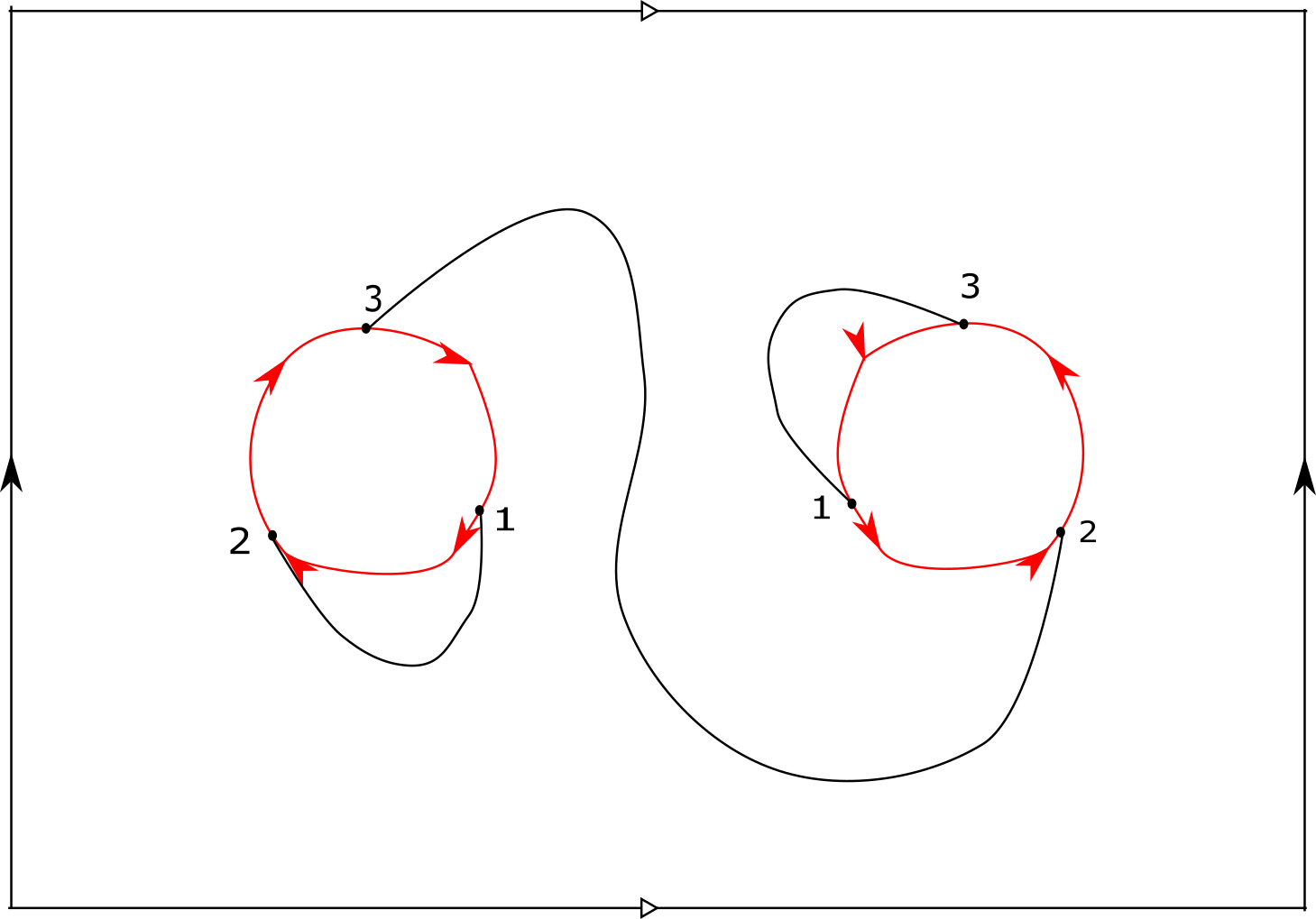}
\caption{$ S_2 $ with a candidate curve $ \alpha $.  Arcs of $ \beta $ cannot connect vertices on the same copy of $ \alpha $.\label{nonF2fig}}
\end{center}
\end{figure}

\indent We will take the red curve to be $ \alpha $ and label three intersections.  To construct $ \beta $, we must construct arcs connecting pairs of vertices; without loss of generality, we may assume the arcs go from vertex 1 to 2, 2 to 3, and 3 to 1.  First, we claim that these arcs must connect vertices on opposite copies of $ \alpha $.  If we connect a pair of vertices on the same copy (1 and 2, say), then necessarily we must connect one of those vertices on the other copy to the third (3 and 1 in the figure).  By sliding 2 backward along $ \alpha $ (or 3 forward), we produce a loop that could be contracted$ - $i.e., $ \alpha $ and $ \beta $ could be isotoped so that they intersect fewer than three times.  Suppose each arc connects a vertex on the left copy to the next-numbered vertex on the right copy (mod 3).

\indent In order to be a filling pair, there must be a ``barrier path" of $ \alpha $-and $ \beta $-arcs parallel to each of the pairs of identified sides of the rectangle:  if there were not such a path parallel to one pair of sides, we could connect two opposite points on the pair perpendicular to them, thus producing an essential scc which is disjoint from $ \alpha \cup \beta $.  But there is no way to construct such barrier paths without having them intersect, implying either that $ \beta $ is not simple, or that $ i(\alpha, \beta)=4 $, a contradiction.
\end{proof}

\begin{theorem}
There exists an $ F_g $ for every $ g \neq 2- $that is, for each such $ g $, there is a filling pair $ (\alpha, \beta) $ on $ S_g $ such that $ i(\alpha, \beta)=2g-1 $.
\end{theorem}

\begin{proof}
The filling permutation $ \sigma=(1,2,3,4) $ defines an $ F_1 $ (in fact, \textit{the unique} $ F_1 $ up to homeomorphism).  Suppose there is an $ F_{2k+1} $ for some $ k \geq 0 $.  By Lemma 3.1, $ F_{2k+1}\ \sharp\ Z=F_{2k+1+2}=F_{2(k+1)+1} $; by induction, there is an $ F_{2k+1} $ for all $ k \geq 0 $.

\indent By Proposition 3.2, there is no $ F_2 $.  The filling permutation

$ \sigma=(1,16,27,10,7,18,15,2,3,20,21,12,11,22,17,4,9,26,19,8,13,28, $

$ 23,6,5,24,25,14) $ defines an $ F_4 $.  Suppose there is an $ F_{2k} $ for some $ k \geq 2 $.  Again, by Lemma 3.1, $ F_{2k}\ \sharp\ Z=F_{2k+2}=F_{2(k+1)} $; by induction, there is an $ F_{2k} $ for all $ 2 \geq 2 $.
\end{proof}

\section{Generalizing $ Z $}

Combinatorially, the essence of the construction described in the previous section is the arithmetic of vertices:  $ Z $ has the right number of intersections so that, when we remove one from each of $ F_g $ and $ Z $, the remaining intersections of $ Z $ are enough to replace the deleted intersection from $ F_g $, plus four more$ - $thus making up the minimum number on $ F_{g+2} $.  In Section 1, we asked if it's possible to generalize this construction:  can we describe a class of surfaces $ Z_k $ (with filling pairs) such that, if we take a connected sum $ F_l\ \sharp\ Z_k $ like that described in Lemma 3.1, the result is an $ F_{l+k} $?

\indent As a first step, we can compute the number of intersections $ Z_k $ must have.

\begin{remark}
Let $ V $ be the number of intersections of the filling pair on $ Z_k $.  When we delete disk neighborhoods from $ Z_k $ and $ F_l $, the intersection numbers are $ V-1 $ and $ 2l-2$, respectively.  When we glue them together, we get $ V-1+2l-2=V+2l-3 $.  Since the total surface is an $ F_{l+k} $, this must equal $ 2(l+k)-1 $.  Thus,

$ V+2l-3=2l+2k-1 $

$ V=2k+2 $.

Further, by the proof of Lemma 2.1,

$ \chi(Z_k)=-(2k+2)+F $

$ 2-2k=-2k-2+F $

$ F=4 $.

\indent So we have our general description of $ Z_k $:  a $ Z_k $ is a closed surface of genus $ k $, together with a filling pair $ (\alpha, \beta) $ with $ i(\alpha, \beta)=2k+2 $, and such that at least one intersection is adjacent to every region. \indent \indent \indent \indent \indent \indent \indent \ \ $ \square $
\end{remark}

\indent We can see that the original surface $ Z $ fits into this class with $ k=2 $.

\begin{example}
We construct a $ Z_3 $ by reverse engineering:  we begin with a union of four disks and describe how to glue their edges together to produce the desired filling pair on a closed surface.  Consider the disk with five handles shown in Figure \ref{Z3fig}.  The handles are attached so that there are two boundary components, each of which intersects every handle.  Topologically, this is the surface $ S_{2,2} $; by identifying the two boundary components, we get an $ S_3 $.

\begin{figure}[ht]
\begin{center}
\includegraphics[scale=.17]{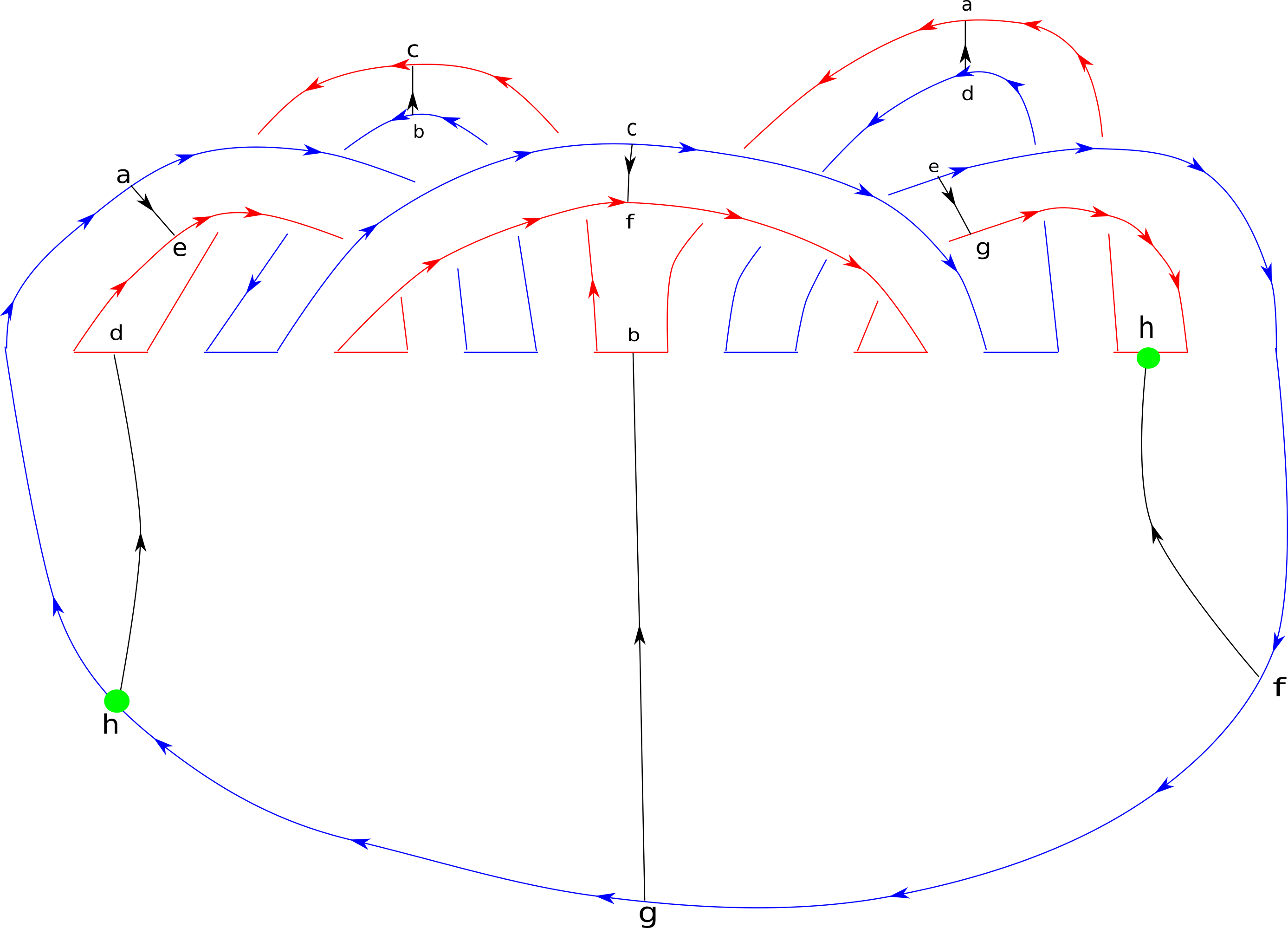}
\caption{A $ Z_3 $ with curves $ \alpha $ (red/blue) and $ \beta $ (black).  Identified points are labeled.  The green vertex is adjacent to all four regions.\label{Z3fig}}
\end{center}
\end{figure}

\indent The identified boundary curves (red/blue) will be $ \alpha $.  To construct $ \beta $, we construct eight arcs connecting the red boundary to the blue boundary (one on each handle and three on the ``feet" of the handles, as shown) so that these arcs partition the surface into four polygonal regions (two rectangles and two 12-gons).  All that remains is to find a way to identify the arcs of the red and blue curves so that the black arcs form a closed curve, and such that one of the intersections is adjacent to all four regions.  In the figure, like-labeled points are identified.

\indent We can describe this $ Z_3 $ by its filling permutation.  If we label the arcs of $ \alpha $ and $ \beta $ so that both curves start and end at the green vertex, we have

$ \sigma_Z=(1,6,25,18)(2,23,28,5,14,27,22,3,12,21,26,15) $

$ (31,24,11,16)(32,13,10,19,20,9,8,29,30,7,4,17). $

\indent This construction can be generalized, at least in principle:  attach $ 2k-1 $ handles to a disk so that there are two boundary components with alternating feet (this implies that both curves intersect each handle).  Construct an arc on each handle connecting the boundary curves, and one each connecting the left-most, center, and right-most feet of the handles to the bottom edge of the disk (by construction, these arcs connect one boundary curve to the other).  This partitions the disk-with-handles into two rectangles and two $ (4k) $-gons.  To complete the construction, we have to find an identification of the boundary curves so that the arcs form a single curve and there is at least one vertex adjacent to all four regions.  It is not apparent that such an identification is always possible. \indent \indent \indent \indent \indent \indent \indent \indent \indent \indent \indent \indent $ \square $
\end{example}

\indent We conclude this section by describing how the connected sum $ F_l\ \sharp\ Z_k $ can be obtained from filling permutations.  The entries of $ \sigma_F $ (the filling permutation for $ F_l $) are in $ \lbrace 1,2, \ldots, 8l-4 \rbrace $ and the entries of $ \sigma_Z $ are in $ \lbrace 1,2, \ldots, 8k+8 \rbrace $.  Let $ E=\lbrace v,w: v \in \lbrace 1,2, \ldots, 8l-4 \rbrace, w \in \lbrace 1,2, \ldots, 8k+8 \rbrace \rbrace $ (we denote $ F_l $-entries by $ v $ and $ Z_k $-entries by $ w $).  We need a function $ A: E \rightarrow \lbrace 1,2, \ldots, 8(l+k)-4 \rbrace $ which relabels all the entries in a consistent fashion.

\indent As noted in Examples 2.1 and 4.1, we adopt the convention of labeling edges of $ Z_k $ so that both curves begin and end at the green vertex.  The other piece of information we will need is which vertex we choose to delete from $ F_l $.  Let $ i $ and $ i+2 $ be the entries corresponding to the positive odd edges adjacent to this vertex (i.e., $ i $ points \textit{into} the vertex and $ i+2 $ points \textit{away}) and let $ j $ and $ j+2 $ be the (entries corresponding to the) positive even edges.  Note that this choice assumes that $ i+2 $ and $ j+2 $ both correspond to positive edges, which need not be true.  We comment on this case below.

\indent Define

\[A=\begin{cases}
v \mapsto v, \ v \leq i \ odd \ or \ v \leq j \ even \\
w \mapsto w+i+4(k+l)-3, \ w \leq 4k+3<4(k+l)-1 \ odd \\
w \mapsto w+j+4(k+l)-4, \ w \leq 4k+4<4(k+l) \ even \\
4(k+l)-1 \mapsto 4(k+l)-1 \\
4(k+l) \mapsto 4(k+l) \\
v \mapsto v+4k, \ i+2 \leq v \leq 4l-3 \ odd \ or \ j+2 \leq v \leq 4l-2 \ even \\
v \mapsto A(v-4l+2)+4(k+l)-2, \ v>4l-2 \\
w \mapsto w+i-4k-5, \ 4k+4<w<8(k+l)-1 \ odd \\
w \mapsto w+j-4k-6, \ 4k+4<w<8(k+l) \ even \\
8(k+l)-1 \mapsto 1 \\
8(k+l) \mapsto 2
\end{cases}
\]

\indent This is our assembly function.  Note that line 2 only applies when $ 4k+3<4(k+l)-1 $; equality implies that $ l=1 $ and, in this case, line 4 tells us to send $ 4k+3 $ (the last positive odd edge of $ Z_k $) to itself (the first negative odd edge of $ F_{k+1} $).  Lines 3 and 5, 9 and 11, and 10 and 12 are similarly related.  As we will see, there is a fundamental difference between the cases $ l=1 $ and $ l>1 $.  Note also that line 6 does not apply to odd edges if $ i=4l-3 $ (nor to even edges if $ j=4l-2 $).  In these cases, line 2 will have $ 4k+3 \mapsto 8(k+l)-3 $ and line 3 will have $ 4k+4 \mapsto 8(k+l)-2- $but the largest entry of the new filling permutation is supposed to be $ 8(k+l)-4 $.  Here, we simply reduce mod $ 8(k+l)-4 $:  $ 8(k+l)-3=1,\ 8(k+l)-2=2 $.  This reflects the cyclicity of the labels:  if $ i=4l-3 $ (the last odd entry for a positive edge) points into the deleted crossing, then 1 (the first odd entry for a positive edge) points away.  Finally, observe that the orientation of $ F_l $-edges is preserved, while the orientation of $ Z_k $-edges is reversed.  This reflects the fact that one subsurface is oriented positively, and the other negatively, with respect to the boundary curves that are identified in the connected sum.

\begin{example}
We will attach the $ Z_3 $ from the previous example to an $ F_3 $ using filling permutations and the assembly function.  Let

$ \sigma_F=(1,2,13,20,7,6,19,14,5,10,11,16,9,8,15,12,3,4,17,18) $.  We apply $ A $ to the entries of $ \sigma_F $ and the \textit{inverse} of $ \sigma_Z $ (again, reflecting the opposite orientations of the two subsurfaces).  We choose $ i=3 $ and $ j=2 $.

$ A(\sigma_F)=(1,2,25,44,19,18,43,38,17,22,23,40,21,20,39,24,3,16,41,42) $.

$ A(\sigma_Z^{-1})=(2,11,28,25)(39,10,7,34,27,6,13,36,29,12,9,24) $

$ (38,35,8,17)(3,26,31,14,15,30,33,4,5,32,37,16) $.

\indent We arranged the cycles of $ \sigma_Z $ so that the first cycle begins with 1 and each subsequent cycle begins with the opposite of the last entry of the previous one.  Notice that, in the relabeled version of this permutation, these first and last entries are duplicated in the relabeled version of $ \sigma_F $, and that entries which begin and end a cycle of $ A(\sigma_Z^{-1}) $ are consecutive in $ A(\sigma_F) $.  This tells us exactly how to ``splice" the cycles of $ A(\sigma_Z^{-1}) $ into $ A(\sigma_F) $.

\indent The result is

$ \sigma=(1,2,11,28,25,44,19,18,43,38,35,8,17,22,23,40,21,20,39,10,7, $

$ 34,27,6,13,36,29,12,9,24,3,26,31,14,15,30,33,4,5,32,37,16,41,42) $. \linebreak
It is straightforward to verify that this is a filling permutation for an $ F_6 $.
\begin{flushright}
$ \square $
\end{flushright}
\end{example}

\section{Decomposition of $ F_g $s as Connected Sums}

We have shwon that we can assemble a larger-genus $ F $ out of the smaller-genus pieces we've described.  Our final question from Section 1 was whether it's possible to reverse our perspective:  given an $ F_g $, is it possible to recognize that it's assembled from such pieces, and is it possible to determine how to disassemble them?  $ F_g $ (with filling pair $ (\alpha, \beta) $) is decomposable as a connected sum $ F_l\ \sharp\ Z_k $ if and only if there exists a separating curve $ \gamma $ with $ i(\alpha, \gamma)=i(\beta,\gamma)=2 $, such that one component of $ F_g \smallsetminus \gamma $ is an $ F_{l,1} $, and the other is a $ Z_{k,1} $.  After cutting them apart, we complete the disassembly by ``patching" each component with a decorated disk $ \bigoplus $, thus replacing the missing intersection and rendering everything with boundary (the surfaces and their arcs) closed.

\indent In order to state our decomposability condition, we must consider an aspect of $ Z_k $ that we haven't yet touched upon.  We have been able to show how many vertices $ Z_k $ must have (which determines how many edges) and how many regions, but we haven't given any indication of how many edges each region can have.  The original $ Z_2 $ has two rectangles and two octagons.  The $ Z_3 $ we constructed in Example 4.1 has two rectangles and two 12-gons, and our hypothetical generalization of that construction produces a surface with two rectangles and two $ (4k) $-gons.  We might conjecture that every $ Z_k $ has such a configuration of regions but, as we will see, this turns out to be false.

\indent In order to capture $ Z_k $ in full generality, then, we give the following description of its regions.  The \textit{type} of a $ Z_k $ is an ordered quadruple $ (r,s,t,u) $ (modulo cyclic permutation), where $ r,s,t $, and $ u $ are even integers greater than 2 such that $ r+s+t+u=8k+8 $.  These are the numbers of sides of each region, and their order indicates the way that they're glued together to form the surface:  each region must have an even number of edges, the smallest possible region is a rectangle, and the total number of edges is $ 8k+8 $.  The order of the entries in (a representative of) the type is the order in which the regions are traversed by a curve encircling the green vertex, oriented clockwise (thus any cyclic permutation of $ (r,s,t,u) $ gives the same type).  We will refer to the type of a decomposition as the type of the associated $ Z_k $.

\begin{theorem}
Let $ F_g $ be a surface of genus $ g>1 $, together with a filling pair $ (\alpha, \beta) $ which intersect $ 2g-1 $ times, and let $ \sigma $ be the filling permutation of $ F_g $.  There exists a decomposition $ F_g=F_l\ \sharp\ Z_k $ of type $ (r,s,t,u) $ if and only if the following two conditions are satisfied:

\begin{enumerate}
\item There exist numbers $ x,a,y,b \in \lbrace 1,2, \ldots, 8g-4 \rbrace $ such that:
\begin{equation}
Q^N \sigma^{r-1}(x)=a
\label{xtoaeqn}
\end{equation}
\begin{equation}
Q^N \sigma^{s-1}(a)=y
\label{atoyeqn}
\end{equation}
\begin{equation}
Q^N \sigma^{t-1}(y)=b
\label{ytobeqn}
\end{equation}
\begin{equation}
Q^N \sigma^{u-1}(b)=x
\label{btoxeqn}
\end{equation}
\begin{equation}
Q^N \tau^{2k+1}(x)=y
\label{xtoyeqn}
\end{equation}
\begin{equation}
Q^N \tau^{2k+1}(a)=b.
\label{atobeqn}
\end{equation}
\item For $ (v,p),(w,q) \in \lbrace (x,r),(a,s),(y,t),(b,u) \rbrace $, $ w=\sigma^i(v) $ for $ i<p-1 $ implies $ \sigma^{p-1}(v)=\sigma^j(w) $ for $ q-1<j<8g-4 $.
\end{enumerate}
\end{theorem}

\indent Note that $ Q^N \tau^{2k+1} $ has order 2 (we have shown that $ Q^N $ and $ \tau $ anticommute, which means that they generate a subgroup isomorphic to $ D_{4g-2} $; the element $ Q^N \tau^{2k+1} $ is a ``reflection," and so it has order 2), so this permutation also sends $ y $ to $ x $ and $ b $ to $ a $.  Necessarily, $ x $ and $ y $ have the same parity (odd/even) and opposite orientations, as do $ a $ and $ b $.

\begin{example}
We will show that this condition detects the decomposition of our $ F_6 $ into the $ F_3 $ and $ Z_3 $ from which we assembled it.  We will also show that it predicts a decomposition into an $ F_1 $ and a $ Z_5 $.

\indent In the first case, we have $ k=3 $ and $ N=22 $.  We see that:

$ Q^{22} \sigma^{11}(3)=38 $

$ Q^{22} \sigma^3(38)=39 $

$ Q^{22} \sigma^{11}(39)=2 $

$ Q^{22} \sigma^3(2)=3 $

$ Q^{22} \tau^7(3)=39 $

$ Q^{22} \tau^7(38)=2 $.  These were precisely the splice points at which we combined the permutations $ A(\sigma_F) $ and $ A(\sigma_Z^{-1}) $, and the powers of $ \sigma $ agree with the type of the $ Z_3 $ we constructed.  Moreover:

$ Q^{22} \sigma^{27}(23)=38 $

$ Q^{22} \sigma^5(38)=1 $

$ Q^{22} \sigma^9(1)=16 $

$ Q^{22} \sigma^3(16)=23 $.  Note that if $ k=g-1 $, then $ \tau^{2k+1}=\tau^{2g-1}=1 $, so we get that $ x $ and $ y $ are opposites, as are $ a $ and $ b $.  If our condition is sufficient, then we should find a decomposition of the surface into an $ F_1 $ and a $ Z_5 $.  Note that the $ Z_5 $ will not have two rectangles and two 20-gons.  $ \square $
\end{example}

\begin{proof}
\textit{Necessity.}

\indent Suppose there exists a curve $ \gamma $ which produces a decomposition of the appropriate type and consider the arc of $ \gamma $ adjacent to the $ r $-gon.  On the $ (8g-4) $-gon, this arc cuts two of the edges of the $ r $-gon in half.  The edges between these ``anchor points" are the edges of the $ r $-gon, but they are also consecutive edges of the $ (8g-4) $-gon, so we can move from one anchor point (call it $ x $) to the other by applying a power of $ \sigma $.  Since we move $ r-1 $ spaces between anchors, the other anchor point is $ \sigma^{r-1}(x) $.  $ Q^N $ takes us to an anchor of the next $ \gamma $-arc.  Let $ a=Q^N \sigma^{r-1}(x) $.  This satisfies (\ref{xtoaeqn}).

\indent Let $ \overline{y} $ be the other anchor of this second $ \gamma $-arc.  By a similar argument, $ \sigma^{s-1} $ takes us from one anchor to the other, but it is not clear whether $ \sigma^{s-1}(a)=\overline{y} $ or vice versa.  If $ y $ and $ \overline{b} $ are the anchors of the third arc (implying that $ b $ and $ \overline{x} $ are the anchors of the final arc), then $ \sigma^{t-1} $ takes us from one of $ y $ and $ \overline{b} $ to the other, and $ \sigma^{u-1} $ takes us from one of $ b $ and $ \overline{x} $ to the other.  These arcs ``cordon off" the regions of $ Z_k $ along the boundary of the $ (8g-4) $-gon (compare Figure \ref{clippeddisksfig}).

\indent Orient $ \gamma $ so that $ x $ is the initial edge of the first arc.  All the arcs must point in the same direction, since the regions of $ Z_k $ all have to be on the same side of $ \gamma $, and we have defined $ x $ and $ \overline{a} $ so that their arc points clockwise.  Thus, $ a, y $, and $ b $ are all initial edges of their respective arcs, implying that (\ref{atoyeqn}), (\ref{ytobeqn}), and (\ref{btoxeqn}) are satisfied.

\indent Now consider $ F_g $ in closed form.  When we remove $ Z_k $ from $ F_l $, we will be removing a string of \textit{consecutive} arcs of $ \alpha $ and $ \beta $.

\begin{figure}[ht]
\begin{center}
\includegraphics[scale=.5]{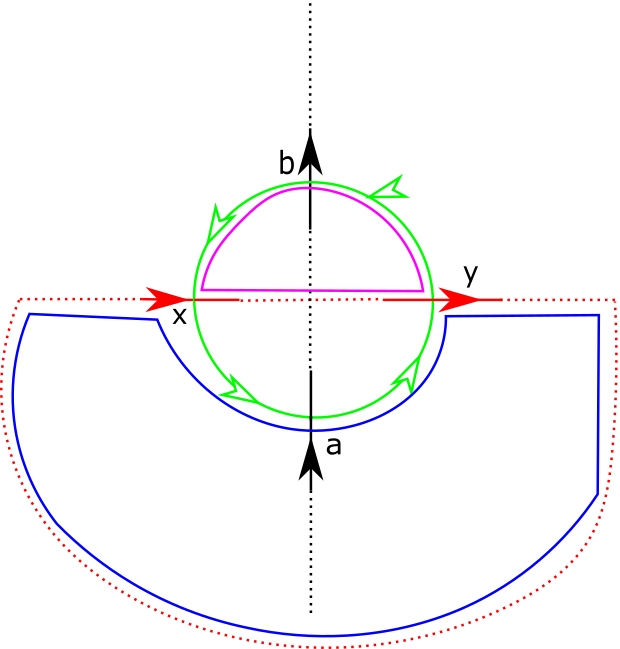}
\caption{The curves $ \alpha $ (red), $ \beta $ (black), and $ \gamma $ (green) on $ F_g $.  The region inside $ \gamma $ is $ Z_k $ and the region outside is $ F_l $.  A region of $ F_g \smallsetminus (\alpha \cup \gamma) $ contained in $ Z_k $ is outlined in purple, and a region contained in $ F_l $ is outlined in blue.\label{gammaregionsfig}}
\end{center}
\end{figure}

\indent Suppose (as in Figure \ref{gammaregionsfig}) that $ x $ is a positive $ \alpha $-edge and that $ a $ is a negative $ \beta $-edge.  $ \gamma $ cuts the arc at $ x $ in half; $ \alpha $ enters $ Z_k $, passes through all the intersections inside it, then exits, so the next half-arc will be at the end of this path.  Since $ Z_k $ contains $ 2k+2 $ arcs, we get from one half-arc to the other by moving forward $ 2k+1 $ spaces along $ \alpha $.  Since $ x $ is a positive edge, $ \tau^{2k+1} $ takes us to the positive side of the other half-arc.  But since this edge is $ \overline{y} $, we have (\ref{xtoyeqn}).  Since $ a $ is a negative edge, the path through $ Z_k $ moves backward along $ \beta $ but, again, $ \tau^{2k+1} $ takes us to the negative side of the other half-arc.  Since this is $ \overline{b} $, we have (\ref{atobeqn}), and therefore Condition 1 is satisfied.

\indent Let $ (v,p), (w,q) \in \lbrace(x,r),(a,s),(y,t),(b,u) \rbrace $.  Each of the edges $ v $ and $ w $ are initial edges of an arc of $ \gamma $.  The $ v $-arc cordons off a $ p $-gon and the $ w $-arc cordons off a $ q $-gon.  If $ w=\sigma^i(v) $ for $ i<p-1 $, then the $ w $-arc originates inside the $ p $-gon.  This implies either that the $ v $-arc and the $ w $-arc intersect (so $ \gamma $ is not simple), or that the $ q $-gon is nested inside the $ p $-gon (so $ \gamma $ ``subcordons" a region which has already been cordoned off).  Both of these are contradictions, implying that there is no such $ v $ and $ w $, so that Condition 2 is satisfied vacuously.
\end{proof}

\indent In order to prove the converse, we assume that the equations in Condition 1 are satisfied, and we use this to define the curve $ \gamma$.  We can describe this curve as a sequence of oriented arcs in the $ (8g-4) $-gon.  We let $ x,a,y $, and $ b $ be the initial edges of the arcs; the terminal edges are $ \overline{a}, \overline{y}, \overline{b} $, and $ \overline{x} $ (respectively).  By \ref{xtoaeqn}, \ref{atoyeqn}, \ref{ytobeqn}, and \ref{btoxeqn}, these arcs cordon off regions of the right size and in the right order to fit the desired type of the decomposition.  In order for this to produce the decomposition we want, we must show that none of these regions overlap, and that $ \gamma $ is separating.  This holds if the opposite of an edge on the left side of the curve is also on the left side.

\begin{proof}
\textit{Sufficiency.}

\indent Condition 2 implies that no two arcs of $ \gamma $ intersect ($ \gamma $ is simple).  This turns out to be sufficient to prove the result.

\indent To see that $ \gamma $ is simple, let $ (v,p), (w,q) \in \lbrace(x,r),(a,s),(y,t),(b,u) \rbrace $, and suppose that $ w=\sigma^i(v) $ for $ i<p-1 $.  As we observed above, this implies that the $ w $-arc originates between the anchor points of the $ v $-arc (since the boundary of the $ (8g-4) $-gon is a circle, any vertex is technically ``between" the anchor points of the $ v $-arc:  we mean here that $ w $ is on the shorter of the two segments connecting the anchor points).  The terminal edge of the $ w $-arc is $ \sigma^{q-1}(w) $; by Condition 2, $ \sigma^{p-1}(v)=\sigma^j(w) $ for $ q-1<j<8g-4 $, which means that, moving clockwise from $ w $, we will reach the terminus of the $ w $-arc before we reach the terminus of the $ v $-arc$ - $i.e., the $ w $-arc is nested inside the $ v $-arc, and therefore they do not intersect.

\indent We now show that, if $ \gamma $ is simple, then $ \gamma $ is separating.  There are two cases:  $ k=g-1 $ and $ k<g-1 $.  In the latter case, there are eight edges that anchor $ \gamma $:  $ x,a,y,b $, and their opposites (and all of these edges must be distinct).  In the former case, as we have noted, we have $ \overline{x}=y $ and $ \overline{a}=b $, so there are only four.  This means that each one is both an initial and a terminal edge, and so each one is cut into thirds by the arcs of $ \gamma $.  The middle thirds of each edge are on the right side of $ \gamma $, and the outer thirds (as well as all the other edges of $ F_g $) are on the left side.  Since the opposite of a middle third is also a middle third, it follows that $ \gamma $ is separating.

\indent Suppose that $ k<g-1 $.  Since $ \gamma $ crosses each edge of the $ (8g-4) $-gon at most once, we can describe its path through the surface in terms of the edges it crosses.  Let $ \rho=(x, \overline{a}, a, \overline{y}, y, \overline{b}, b, \overline{x}) $.  This permutation is defined by choosing an arc of $ \gamma $ and recording in a cycle the numbers of the edges it crosses in the order in which they occur:  from an initial edge to its terminal edge, from a terminal edge to its opposite (the next initial edge), and so on, until the cycle is completed.

\indent Let $ \tau=\tau_\alpha \tau_\beta $ and $ Q^N=\phi_\alpha \phi_\beta $, where we have factored each permutation into a portion that moves only edges of $ \alpha $ and one that moves only edges of $ \beta $.  Recall that $ \tau $ acts by cycling positive edges of $ \alpha $ and $ \beta $ in the forward direction, and negative edges in the backward direction; we can obtain from $ \rho $ a permutation which acts similarly on the edges of $ \gamma $:  $ \rho'=(x,a,y,b)(\overline{x}, \overline{b}, \overline{y},\overline{a}) $.  Note that $ \rho $ and $ \rho' $ are related by the equation $ \rho'=[\phi_\gamma, \rho] $, where $ \phi_\gamma $ is made up of the 2-cycle factors of $ Q^N $ that move only edges appearing in $ \gamma $ (in fact, $ \phi_\gamma \rho $ is the initial-edge cycle of $ \rho' $ and $ \phi_\gamma \rho^{-1} $ is the terminal-edge cycle; note that they commute with one another).

\indent We will show that the permutation $ \tau_\alpha \phi_{\alpha \gamma} \rho' $ (where $ \phi_{\alpha \gamma} $ consists of the factors of $ Q^N $ that move only edges that appear in \textit{both} $ \alpha $ and $ \gamma $) traces out the boundaries of the complementary regions of $ \alpha \cup \gamma $ (and likewise for $ \beta $).

\indent To trace out one boundary component of a complementary region, we choose a starting edge and move along $ \alpha \cup \gamma $, keeping the same region to the right of $ \alpha \cup \gamma $.  Suppose $ x $ is an $ \alpha $-edge.  Starting at $ x $ (an initial edge) and moving along $ \gamma $ takes us to $ \overline{a} $, which is an edge of $ \beta $ (note that, since $ \gamma $ doesn't intersect itself, the boundary curves are only formed by $ \gamma $-edges intersecting $ \alpha $-edges).  We want to pass through this edge and proceed along the next arc of  $\gamma $.  $ \rho' $ sends $ x $ to $ a $ (the next initial edge); $ \phi_{\alpha \gamma} $ and $ \tau_\alpha $ fix $ a $ (since it's not an edge of $ \alpha $), and so $ a $ is the second entry of this cycle of $ \tau_\alpha \phi_{\alpha \gamma} \rho' $.

\indent $ \rho'(a)=y $, the next initial edge.  In order to reach this edge, though, we have crossed over $ \alpha $; instead, we would like to follow $ \alpha $ (again, keeping the same region on the right).  So we return to the terminal edge on the other side of $ \alpha $ by sending $ y $ to its opposite, which is $ \phi_{\alpha \gamma}(y) $.  If $ \overline{y} $ is a positive edge, we should follow $ \alpha $ in the forward direction and, if it's a negative edge, we should follow in the backward direction:  either way, the next edge along the boundary will be $ \tau_\alpha(\overline{y}) $.

\indent Proceeding in this fashion, we trace out a sequence of $ (\alpha \cup \gamma) $-edges, always keeping the same region on the right.  If the next edge is an initial edge of $ \gamma $, we should follow it in the forward direction and, if it's a terminal edge, we should follow it in the backward direction$ - $in either case, $ \rho' $ moves us to the appropriate edge.  If the edge doesn't intersect $ \gamma $, then both $ \rho' $ and $ \phi_{\alpha \gamma} $ fix it, and $ \tau_\alpha $ advances it in the appropriate direction along $ \alpha $.

\indent If we move along an edge of $ \gamma $ and come to a $ \beta $-edge, we should keep moving ($ \phi_{\alpha \gamma} $ and $ \tau_\alpha $ fix this edge, and then $ \rho' $ sends it on along $ \gamma $).  If we cross an $ \alpha $-edge (initial or terminal), we should back up to the previous edge (terminal or initial, via $ \phi_{\alpha \gamma} $), and then move along $ \alpha $ according to $ \tau_\alpha $.  When we complete a circuit of one boundary component, we will have completed a cycle of $ \tau_\alpha \phi_{\alpha \gamma} \rho' $.  For each boundary component of $ \alpha \cup \gamma $, there is one cycle of $ \tau_\alpha \phi_{\alpha \gamma} \rho' $.  Observe that it fixes all edges of $ \beta $ that don't intersect $ \gamma $.

\indent It's important to note that $ \tau_\alpha \phi_{\alpha \gamma} \rho' $ only ``traces" the boundaries of the complementary regions:  not every edge of the boundary will appear in one of its cycles and some edges will appear that are not on the boundary.  For instance, the edge $ a $ belongs to $ \beta $, so it is not on the boundary of $ \alpha \cup \gamma $.  In general, a $ \beta $-entry will appear each time we cross an arc of $ \beta $ and an $ \alpha $-entry will be missing each time we switch from following $ \gamma $ to following $ \alpha $.

\indent If we again suppose $ x $ is an $ \alpha $-edge, we claim that the cycles of $ \tau_\alpha \phi_{\alpha \gamma} \rho' $ are:

\begin{equation}
(x,a,\tau(\overline{y}),\tau^2(\overline{y}),\ldots, \tau^{2l-3}(\overline{y}))
\label{xcycle}
\end{equation}
\begin{equation}
(y,b,\tau(\overline{x}),\tau^2(\overline{x}),\ldots,\tau^{2l-3}(\overline{x}))
\label{ycycle}
\end{equation}
\begin{equation}
(\overline{x},\overline{b},\tau(y),\tau^2(y),\ldots,\tau^{2k}(y))
\label{xbarcycle}
\end{equation}
\begin{equation}
(\overline{y},\overline{a},\tau(x),\tau^2(x),\ldots,\tau^{2k}(x)).
\label{ybarcycle}
\end{equation}

\indent The computations of (\ref{xcycle}) and (\ref{ycycle}) are similar, as are the computations of (\ref{xbarcycle}) and (\ref{ybarcycle}), so we will show the computations of (\ref{xcycle}) and (\ref{xbarcycle}).

\indent $ \rho'(x)=a $, which is fixed by the other two permutations; $ \rho'(a)=y $, and $ \phi_{\alpha \gamma}(y)=\overline{y} $, so the next entry is $ \tau(\overline{y}) $.  It remains to show that none of the edges $ \tau^i(\overline{y}) $ is an entry of $ \rho' $ for $ 1 \leq i \leq 2l-3 $.  $ y $ has the same parity as $ \overline{y} $; since $ \tau $ preserves parity of edges, if $ \tau^i(\overline{y}) $ belongs to $ \rho' $, it must be an edge with the same parity as $ y $.  However, $ x $ and $ y $ have opposite orientations; since $ \tau $ preserves orientation, the only possibilities are $ x $ and  $ \overline{y} $.  $ \overline{y}=\tau^{2k+1}(x) $, so $ x=\tau^{-(2k+1)}(\overline{y})=\tau^{2l-2}(\overline{y}) $ (since $ 2g-1=2k+1+2l-2 $), and $ \overline{y}=\tau^{2g-1}(\overline{y}) $.  Since $ l<g,\ i<2l-2<2g-1 $, so $ \tau^i(\overline{y}) $ cannot equal $ x $ or $ \overline{y} $.  Note that $ \tau^{2l-2}(\overline{y})=\tau^{2g-1}(x)=x $, completing the cycle.

\indent The first three entries of (\ref{xbarcycle}) are determined exactly as those of (\ref{xcycle}).  All that remains to show is that none of the edges $ \tau^i(y) $ is an entry of $ \rho' $ for $ 1 \leq i \leq 2k $.  If there were such an entry, it could only be $ y $ or $ \overline{x} $.  $ \overline{x}=\tau^{2k+1}(y) $ and $ y=\tau^{2g-1}(y) $.  Since $ k<g-1,\ i<2k+1<2g-1 $, so $ \tau^i(y) $ cannot equal $ y $ or $ \overline{x} $.  $ \tau^{2k+1}(y)=\overline{x} $, completing the cycle.

\indent To show that $ \gamma $ is separating, we must show that, for any edge of the $ (8g-4) $-gon that is not an entry of $ \rho' $, its opposite lies on the same side of $ \gamma $.  To do this, we examine the cycles of $ \tau_\alpha \phi_{\alpha \gamma} \rho' $ (and $ \tau_\beta \phi_{\beta \gamma} \rho' $) and show that the opposite of any non-$ \rho' $ entry is an entry of a like cycle$ - $that is, entries of (\ref{xcycle}) and (\ref{ycycle}) are opposites and entries of (\ref{xbarcycle}) and (\ref{ybarcycle}) are opposites.  This implies that it's impossible to travel from one side of $ \gamma $ to the other  by only crossing edges of $ \alpha $ (or $ \beta $) and, therefore, $ \gamma $ is separating.

\indent Let $ 1 \leq i \leq 2k $ and consider $ \tau^i(y) $ (an entry of (\ref{xbarcycle})).  $ Q^N $ anticommutes with $ \tau $; thus, $ Q^N\tau^i(y)=\tau^{-i}(\overline{y})=\tau^{-i} \tau^{2k+1}(x)=\tau^{2k+1-i}(x) $.  Since $ -2k \leq -i \leq -1,\ 1 \leq 2k+1-i \leq 2k $, and so $ Q^N \tau^i(y) $ is an entry of (\ref{ybarcycle}).  Note that the $ \tau $-powers of $ x $ and $ y $ are in a dual relationship to one another:  the opposite of $ \tau(y) $ is $ \tau^{2k}(x) $, the opposite of $ \tau^2(y) $ is $ \tau^{2k-1}(x) $, and so on.  This implies that the opposite of any non-$ \rho' $ entry of (\ref{ybarcycle}) is a non-$ \rho' $ entry of (\ref{xbarcycle}).

\indent Now let $ 1 \leq i \leq 2l-3 $.  Then $ \tau^i(\overline{y}) $ is an entry of (\ref{xcycle}).  $ Q^N \tau^i(\overline{y})=\tau^{-i}(y)=\tau^{-i} \tau^{-2k-1}(\overline{x})=\tau^{2g-1-2k-1-i}(\overline{x})=\tau^{2l-2-i}(\overline{x}) $.  Since $ -2l+3 \leq -i \leq -1,\ 1 \leq 2l-2-i \leq 2l-3 $, so $ Q^N \tau^i(\overline{y}) $ is an entry in (\ref{ycycle}).  Again, we have the same dual relationship, proving the claim.  A similar argument applies for $ \beta $.

\indent Note that, since $ \gamma $ is separating, its arcs cannot be nested, since this would imply there is some region of the $ (8g-4) $-gon which is on both sides of $ \gamma $.  Thus $ \gamma $ realizes a decomposition of type $ (r,s,t,u) $.

\indent This concludes the proof of sufficiency.
\end{proof}

\begin{proposition}
In the case $ k=g-1 $ above, Condition 1 implies Condition 2.
\end{proposition}

\begin{proof}
As noted above, in this case, $ \overline{x}=y $ and $ \overline{a}=b $.  We show that no two arcs of $ \gamma $ intersect.

\indent The arc of $ \gamma $ that originates at $ x $ terminates at $ b $.  We will show that neither of the edges $ a $ or $ y $ can lie between $ x $ and $ b $; a similar argument will show that, for any pair of anchor points, neither of the other two edges can lie between them and, therefore, none of the arcs of $ \gamma $ cross.  Suppose $ a $ lies between $ x $ and $ b- $that is, $ a=\sigma^i(x),\ 1 \leq i \leq r-1 $ (odd).  Then $ \sigma^{s-1}(a)=\sigma^{i+s-1}(x) $; but then $ x=\sigma^{s-1}(a)=\sigma^{i+s-1}(x) $.  $ i+s-1<r+s-2 \leq 8k-2=8g-10 $ (since the $ t $- and $ u $-gons must be at least rectangles and so, at most, $ r+s=8k $).  But since the order of $ \sigma $ is $ 8g-4 $, this is a contradiction.

\indent Now suppose $ y=\sigma^i(x),\ 2<i \leq r-2 $ (even).  $ \sigma^{t-1}(y)=a=\sigma^{-(s-1)}(x) $, so

$ \sigma^{t+s-2}(y)=x $

$ \sigma^{t+s+i-2}(y)=y $.  But $ t+s+i-2 \leq r+s+t-4<8g-4 $, again a contradiction.

\indent As we showed in the Proof of Sufficiency, if $ \gamma $ is simple, then $ \gamma $ is separating, so in fact Condition 2 is satisfied vacuously.
\end{proof}

\begin{remark}
Just as we can assemble the filling permutations of an $ F_l $ and $ Z_k $ into an $ F_g $, we can use Theorem 5.1 to extract permutations corresponding to the constituent surfaces from the filling permutation for $ F_g $, and then use an appropriate relabeling function to recover the actual filling permutations for $ F_l $ and $ Z_k $.  As might be expected, we need two extraction processes and two relabeling functions:  one each for each of the cases $ k=g-1 $ and $ k<g-1 $.

\indent In the latter case, we can extract the $ Z_k $-permutation by using equations (\ref{xtoaeqn}), (\ref{atoyeqn}), (\ref{ytobeqn}), and (\ref{btoxeqn}) from Condition 1.  Let

$ \check{\sigma}=(x, \sigma(x), \ldots, \overline{a})(a, \sigma(a), \ldots, \overline{y}) $

$ (y, \sigma(y), \ldots, \overline{b})(b, \sigma(b), \ldots, \overline{x}) $.

\indent The cycles of $ \check{\sigma} $ have lengths $ r, s, t $, and $ u $, respectively.  Let $ i \in \lbrace x,a,y,b \rbrace $ be the positive odd edge and $ j $ the positive even edge.  Arrange the cycles of $ \check{\sigma} $ so that the first one ends with the odd edge that is \textit{not} $ i $ and the first entry of each subsequent cycle is the opposite of the last entry of the cycle before it.  Let $ \hat{\sigma} $ be the filling permutation for $ F_g $ with the interior entries of $ \check{\sigma} $ (i.e., the ones between the first and last entries) deleted.  We now define the appropriate disassembly function to relabel the entries of these two permutations.

\[D=\begin{cases}
v \mapsto v-max \lbrace0, i-4(g-k)+3 \rbrace,\ max \lbrace1, i-4(g-k)+4 \rbrace \leq v \leq i \ odd \\
v \mapsto v-max \lbrace0, j-4(g-k)+2 \rbrace,\ max \lbrace2, j-4(g-k)+4\rbrace \leq v \leq j\ even \\
v \mapsto v-4k,\ i+4k+2 \leq v \leq 4g-3\ odd\ or\ j+4k+2 \leq v \leq 4g-2\ even \\
w \mapsto w-i+4k+5,\ i \leq w \leq min \lbrace i+4k+2, 4g-3 \rbrace\ odd \\
w \mapsto w-i+4(k+g)+3,\ 1 \leq w \leq i-4(g-k)+4\ odd \\
w \mapsto w-j+4k+6,\ j \leq w \leq min \lbrace j+4k+2, 4g-2 \rbrace\ even \\
w \mapsto w-j+4(k+g)+4,\ 2 \leq w \leq j-4(g-k)+4\ even \\
v \mapsto D(v-4g+2)+4(g-k)-2, \ v>4g-2 \\
w \mapsto D(w-4g+2)-4k-4,\ w>4g-2
\end{cases}
\]

\indent Again, we denote the entries of the permutation corresponding to $ F_l $ by $ v $ and entries of the one corresponding to $ Z_k $ by $ w $.  The choice of notation for these permutations is meant to suggest the way one surface is attached to the other.  If we think of the ``points" of both the ``check" and ``hat" as representing the vertices that are cut out of each surface, then $ \check{\sigma} $ corresponds to the surface that is attached to the other, and $ \hat{\sigma} 
$ corresponds to the surface that receives the attachment.

\indent The sequence of odd $ w $s begins at $ i $ and includes the next $ 2k+1 $ consecutive \textit{positive} edges (since we want to skip over even entries, the last one will be $ i+4k+2 $).  There are two possibilities:  either we reach the end of this sequence before running out of positive edges, or we don't.  In the former case, all the positive odd entries from 1 up to (and including) $ i $ will be $ v $s, the entries from $ i $ up to $ i+4k+2 $ are $ w $s, and the entries from $ i+4k+2 $ to $ 4g-2 $ are $ v $s:  the sequence of $ Z_k $-entries is ``sandwiched" between sequences of $ F_l $-entries.  In the latter case, when we run out of positive edges ($ i \leq w \leq 4g-3 $), we wrap around to the beginning, continue until we complete the sequence of $ w $s ($ 1 \leq w \leq i-4(g-k)+4 $), and then begin the sequence of $ v $s ($ i-4(g-k)+4 \leq v \leq i $):  here, the $ F_l $-sequence is sandwiched between $ Z_k $-sequences.  Accordingly, lines 5 and 7 will only be used if $ 4g-3<i+4k+2 $, and line 3 will only be used if $ i+4k+2 \leq 4g-3 $.  A similar observation applies to even entries.

\indent In the case $ k=g-1 $, we will only have $ \check{\sigma} $, since we can take $ \hat{\sigma} $ to be $ (1,2,3,4) $.  The difficulty is that $ F_g $ has $ 8g-4 $ edges, while $ Z_{g-1} $ has $ 8(g-1)+8=8g $ edges.  The edges $ x,a,y $, and $ b $ each serve as both initial and terminal edges.  In order to construct $ \check{\sigma} $, then, we have to add duplicate copies of each of these edges.  Let $ \tilde{x}, \tilde{a}, \tilde{y} $, and $ \tilde{b} $ denote the left end (with respect to orientation) of their respective edges, and let $ x,a,y $, and $ b $ denote the right end.  A positive edge is decorated when it appears as a terminal edge and undecorated when it's an initial edge, and vice versa for negative edges.  Construct $ \check{\sigma} $ just as described above, decorating edges as appropriate.  Let $ i $ and $ j $ be as above.  Then

\[\tilde{D}=\begin{cases}
w \mapsto w-i+4k+5,\ i \leq w \leq 4k+1 \ odd \\
w \mapsto w-i+8k+7,\ 1 \leq w \leq i-2\ odd \\
\tilde{i} \mapsto 8k+7 \\
w \mapsto w-j+4k+6,\ j \leq w \leq 4k+2\ even \\
w \mapsto w-j+8k+8,\ 2 \leq w \leq j-2\ even \\
\tilde{j} \mapsto 8k+8 \\
w \mapsto w-i-4k-1,\ i+4k+2 \leq w \leq 8k+3\ odd \\
w \mapsto w-i+1,\ 4k+3 \leq w \leq i+4k\ odd \\
\overline{\tilde{i}} \mapsto 4k+3 \\
w \mapsto w-j-4k,\ j+4k+2 \leq w \leq 8k+4\ even \\
w \mapsto w-j+2,\ 4k+4 \leq w \leq j+4k\ even \\
\overline{\tilde{j}} \mapsto 4k+4
\end{cases}
\]
\begin{flushright}
$ \square $
\end{flushright}
\end{remark}

\begin{example} For the $ F_6 $ above, we had

$ x=23 $

$ a=38 $

$ y=1 $

$ b=16 $. \\

\indent Then $ \check{\sigma}=(\tilde{38},35,8,17,22,23)(1,2,11,28,25,44,19,18,43,38)(16,41,42,\tilde{1}) $

$ (\tilde{23},40,21,20,39,10,7,34,27,6,13,36,29,12,9,24,3,26,31,14,15,30, $

$ 33,4,5,32,37,\tilde{16}) $.\\

$ i=1, j=16 $.  After relabeling, we get

$ \tilde{D}(\check{\sigma})^{-1}=(1,32,41,40,13,24)(48,15,18,29,36,11,16,39,46,9,12,27,10, $

$ 33,44,7,22,37,38,5,20,31,42,17,30,45,4,23)(47,6,19,26) $

$ (2,21,28,43,8,3,14,35,34,25) $, a filling permutation for a $ Z_5 $.

\indent We originally constructed this $ F_6 $ as $ F_3\ \sharp\ Z_3 $, and we just performed a decomposition as $ F_1\ \sharp\ Z_5 $.  This suggests that we should find a decomposition $ F_3=F_1\ \sharp\ Z_2 $.  Indeed:

$ Q^{10} \sigma_F^7(1)=4 $

$ Q^{10} \sigma_F^3(4)=11 $

$ Q^{10} \sigma_F^7(11)=14 $

$ Q^{10} \sigma_F^3(14)=1 $.  Since $ k=g-1 $, we have two pairs of opposite edges. \\
\\

Then $ \check{\sigma_F}=(\tilde{14}, 5, 10, 11)(1, 2, 13, 20, 7, 6, 19, 14)(4, 17, 18, \tilde{1}) $

$ (\tilde{11}, 16, 9, 8, 15, 12, 3, \tilde{4}) $ \\

After relabeling, we have $ \zeta'=\tilde{D}(\check{\sigma_F})^{-1}=(1, 20, 17, 12) $

$ (24, 15, 10, 5, 18, 21, 4, 11)(23, 6, 7, 14)(2, 9, 16, 19, 8, 3, 22, 13) $.

\indent We are now able to answer the first of our questions from Section 1.  $ \zeta' $ defines some surface $ Z_2 $, and

$ \zeta=(1, 10, 15, 20, 17, 22, 3, 12)(24, 5, 18, 11)(23, 16, 9, 6, 7, 4, 21, 14) $

$ (2, 19, 8, 13) $ defines our original $ Z $-piece.  We will show that the two surfaces are not homeomorphic.  Both of them have the same type ($ (8,4,8,4) $), and both have two vertices that are adjacent to all four regions.  We can identify these vertices by their left edges$ - $that is, the edges to the left of the identified points on the union of polygons.  By the design of the labeling process, the green vertices of both surfaces will have left edges $ \lbrace 11, 12, 13, 14 \rbrace $.  On $ Z_2 $, the other vertex has left edges $ \lbrace 5, 6, 19, 20 \rbrace $.  Since a homeomorphism must preserve adjacency, the edges $ \lbrace 11, 12, 13, 14 \rbrace $ of $ Z $ must either be mapped to the identical set of edges of $ Z_2 $ (in some order), or to $ \lbrace 5, 6, 19, 20 \rbrace $ (in some order).  No such reassignment can be realized by conjugation by an element of the group of twisting permutations $ T<\Sigma_{24} $ defined in Proposition 2.1, and so the surfaces $ Z $ and $ Z_2 $ are not homeomorphic.  Notice that these three decompositions ($ F_6=F_3\ \sharp\ Z_3 $, $ F_3=F_1\ \sharp\ Z_5 $, and $ F_3=F_1\ \sharp\ Z_2 $) imply that $ Z_5=Z_2\ \sharp\ Z_3 $:  we can produce higher-genus $ Z $s by connect-summing lower-genus $ Z $s.  \indent \indent \indent \indent \indent \indent \indent \indent \indent \indent \indent \indent \indent \indent \indent \indent \indent \indent \indent \indent \ \ $ \square $

\end{example}

\indent We close with the observation that assembly and disassembly aren't perfectly inverse processes.  For example, if we apply $ A $ to $ \tilde{D}(\check{\sigma}) $, and then splice it's cycles into a relabeled $ (1,2,3,4) $, we get

$ \sigma'=(1,10,11,36,25,30,19,4,43,24,35,16,17,8,23,26,21,6,39,18,7, $

$ 42,27,14,13,44,29,20,9,32,3,34,31,22,15,38,33,12,5,40,37,2,41,28) $.\\

For comparison, we recall that the original filling permutation for this $ F_6 $ was 

$ \sigma=(1,2,11,28,25,44,19,18,43,38,35,8,17,22,23,40,21,20,39,10,7, $

$ 34,27,6,13,36,29,12,9,24,3,26,31,14,15,30,33,4,5,32,37,16,41,42) $.

Though the two permutations are different, their odd entries are the same and their even entries can be cyclically permuted to one another (in fact, $ \delta^{-1} \sigma' \delta=\sigma $, where $ \delta $ is as defined in Proposition 2.1).  In assembling an $ F_g $ from an $  F_1 $ and $ Z_{g-1} $, the only choice of vertex in $ F_1 $ gives $ i=1 $ and $ j=2 $ (in the definition of $ A $), which means that we must have the same values for $ i $ and $ j $ in the definition of $ \tilde{D} $.  In disassembling $ \sigma $, we indeed had $ i=1 $ (which is why the odd edges of $ \sigma $ and $ \sigma' $ agree), but $ j=16 $.  In general, if $ \sigma $ is a filling permutation that admits a decomposition, and $ \sigma' $ is a filling permutation produced by disassembling and then reassembling $ \sigma $, then $ \sigma $ and $ \sigma' $ are conjugate by a power of $ \delta $ (cycling the labels of $ \beta $), a power of $ \kappa $ (cycling the labels of $ \alpha $), or both.


\begin{thebibliography}{9}
\bibitem[AH]{AH} T. Aougab, S. Huang.  \textit{Minimally Intersecting Filling Pairs on Surfaces}.  Algebraic and Geometric Topology 15, 903-932, 2015.
\bibitem[F]{F} B. Farb.  \textit{Relatively Hyperbolic Groups}.  Geometric and Functional Analysis 8, 1-31, 1998.
\bibitem[Ha]{Ha} W. J. Harvey.  \textit{Boundary structure of the modular group.  Riemann surfaces and related topics: Proceedings of the 1978 Stony Brook Conference (State Univ. New York, Stony Brook, N.Y., 1978)}, 245–251, Ann. of Math. Stud., 97, Princeton Univ. Press, Princeton, N.J., 1981. ISBN 0-691-08264-2.
\bibitem[He]{He} J. Hempel.  \textit{3-manifolds  as  viewed  from  the  curve  complex}.  Topology, 40(3):631-657, 2001.
\bibitem[MM]{MM} H. A. Masur and Y. N. Minsky.  \textit{Geometry of the complex of curves. I.
Hyperbolicity}.  Invent. Math., 138(1):103-149, 1999.
\end{thebibliography}
\end{document}